\documentclass{article}
\usepackage[top=3cm,bottom=3cm,left=3cm,right=2.5cm]{geometry}
\usepackage{amsfonts}\usepackage{amssymb}
\usepackage{graphicx}\usepackage{amsmath}
\usepackage{subfigure}
\usepackage{abstract}
\usepackage{color}
\usepackage{cite}
\usepackage{graphics}
\usepackage{amsmath}
\usepackage{epstopdf}
\usepackage{amsthm}
\usepackage{float}

\newtheorem{theorem}{Theorem}[section]
\newtheorem{lemma}{Lemma}[section]

\newtheorem{remark}{Remark}[section]

\numberwithin{equation}{section}
\providecommand{\keywords}[1]{\textbf{Key words.} #1}

\title{The a posteriori error estimates and an adaptive algorithm of the FEM for transmission eigenvalues for anisotropic media}

\author{Shixi Wang$^{1}$,  Hai Bi$^{1}$, Yidu Yang$^{1,}$\footnote{Corresponding author. Email address: ydyang@gznu.edu.cn}
 \\\\
{\small $^1$ School of Mathematical Sciences,
Guizhou Normal University,Guiyang,  $550001$,  China}
}

\begin{document}
\date{}
\maketitle
\begin{abstract}
The transmission eigenvalue problem arising from the inverse scattering theory
is of great importance in the theory of qualitative methods and in the practical applications.
In this paper, we study the transmission eigenvalue problem for anisotropic inhomogeneous media in $\Omega\subset \mathbb{R}^d\,(d=2,3)$.
Using the $\mathbb{T}$-coercivity and the spectral approximation theory,
we derive an a posteriori estimator of residual type
and prove its effectiveness  and reliability for eigenfunctions.
In addition, we also prove the reliability of the estimator for transmission eigenvalues.
The numerical experiments indicate our method is efficient and can reach the optimal order $DoF^{-2m/d}$ by using piecewise polynomials of degree $m$ for real eigenvalues.
\end{abstract}

\section{Introduction}

\indent The transmission eigenvalue problem arising from the inverse scattering theory for inhomogeneous media
 is of great importance in the proof of the unique determination of an inhomogeneous media, and transmission
eigenvalues can be reconstructed from the scattering data and used to obtain physical
properties of the unknown target \cite{cakoni2008,cakoni2013,colton1998}. So,
 the computation for transmission eigenvalues becomes an attractive issue in academic circle since the first paper by Colton, Monk and Sun \cite{colton2010}.\\
\indent In this paper, we consider the following transmission eigenvalue
problem: Find $\mathrm{k}\in\mathbb{C},(w,v)\in H^1(\Omega)\times H^1(\Omega)$ such that
\begin{equation}\label{s1.1}
\left\{
  \begin{array}{ll}
   \nabla\cdot(A\nabla w) + \mathrm{k}^2nw = 0 ~~&\mathrm{in}\,~ \Omega, \\
   \triangle v + \mathrm{k}^2v = 0~~&\mathrm{in}\,~ \Omega, \\
   w-v =0  ~~~~~~~~~~~~~&\mathrm{on} \,~\partial\Omega,\\
   \frac{\partial w}{\partial \nu_A}-\frac{\partial v}{\partial \nu}=0~~~~~~~~~~~~~&\mathrm{on}\,~ \partial\Omega,
  \end{array}
\right.
\end{equation}
where $\Omega\subset \mathbb{R}^d\,(d=2,3)$ is a bounded domain with Lipschitz boundary $\partial\Omega$,
$A(x)$ is a real-valued symmetric matrix function and $n(x)$ is the index of refraction, $\nu$ is the unit outward normal to the boundary $\partial \Omega$ and $\frac{\partial w}{\partial \nu_A} = \nu\cdot(A\nabla w)$.\\
\indent As for the problem (\ref{s1.1}), most existing numerical researches only cover the isotropic media, i.e., $A(x)=I$
(see \cite{cakoni2016,sunzhou2016} and the references therein, and see articles
\cite{an2016,zeng2016,yang2016,geng2016,chen2017,li2017,han2017,kleefeld2017,kleefeld2018a,kleefeld2018b,camano2018,mora2018,lih2018,yang2020} and so on). In the case of anisotropic inhomogeneous media ($A(x)\not=I$), due to the complexity of the problem being neither elliptic nor self-adjoint, few numerical treatments have been addressed.
Ji et al. \cite{ji2013} first propose a mixed finite element method and a multilevel method but without convergence proof,
Xie et al. \cite{xie2017} proved the convergence of the mixed finite element method and gave a multilevel correction method
using $\mathbb{T}$-coercivity, and Gong et al. \cite{gong2020} formulated the problem
as a eigenvalue problem of a holomorphic Fredholm operator function
and proposed a finite element discretization.
\\
When studying the transmission equation/eigenvalue problems, the  $\mathbb{T}$-coercivity is important and useful (see \cite{dhia2010,dhia2011,dhia2012,chesnel2013,wu2013,xie2017}) which is equivalent to the $\inf$-$\sup$ condition (see \cite{ciarlet2012}).
In this paper, using the $\mathbb{T}$-coercivity we study the a posteriori error estimates of residual type of mixed finite element method
and an adaptive method for the problem (\ref{s1.1}) when $A(x)\not=I$, and
our work has the following features:\\
\indent (1)~Under the adaptive computation mode, it generally needs to solve the primal and the adjoint problems at the same time when solving non-selfadjoint eigenvalue problems (see, e.g., \cite{gedicke,gasser}). Based on the mixed formulation which was proposed in \cite{dhia2011}, 
we show that $\lambda$ is the eigenvalue of the primal eigenvalue problem with the corresponding generalized eigenfunction $\mathbf{U}^{j}$ of order $j\geq 1$ if and only if $\lambda$ is the eigenvalue of the adjoint eigenvalue problem and $\mathbf{U}^{j}$ is a generalized eigenfunction of order $j\geq 1$ corresponding to $\lambda$ (see Theorem 2.1).
Due to this fact, in our adaptive method it has no need to analyze and solve
the adjoint eigenvalue problem but the primal problem only.\\
\indent (2)~ Xie and Wu \cite{xie2017} first presented an a priori error estimate of the mixed finite element method in terms of the gap of the generalized eigenfunction space.
Based on \cite{xie2017}, this paper further studies the a priori error estimate.
Note that if $(w,v)$ is the solution of the source problem corresponding to the problem (\ref{s1.1}) then $(w,v)\in H^{1+r}(\Omega)\times H^{1+r}(\Omega)$ (see \cite{grisvard1985}). However, to the best of our knowledge, the following regularity estimate
is not available in general:
\begin{equation}\label{s1.2}
\|w\|_{H^{1+r}}+\|v\|_{H^{1+r}}\leq C_{\Omega}(\|f\|_{L^{2}(\Omega)}+\|g\|_{L^{2}(\Omega)}),
\end{equation}
where $(f,g)\in L^{2}(\Omega)\times L^{2}(\Omega)$ is the right-hand side of the source problem and $C_{\Omega}$ is the a prior constant.
Hence, we use the approximation property of the finite element space and the fact that the point-wise convergence on a compact set implies the uniform convergence to prove that the discrete solution operator converges to the solution operator according to the operator norm.
Then we employ the $\mathbb{T}$-coercivity and the spectral approximation theory to prove that the approximate eigenpairs converge to the exact eigenpairs, and prove the error estimate of eigenfunctions in $L^2(\Omega)\times L^2(\Omega)$ norm is a small quantity of higher order compared with the error estimate in $H^1(\Omega)\times H^1(\Omega)$ norm.\\
\indent (3)~In practical finite element computations, it is desirable to carry out the
computations in adaptive fashion (see, e.g., \cite{ainsworth,babuska1978,brenner1,morin,shi2013,verfurth1996}
 and the references cited therein).
\cite{wu2013} gave an a posteriori error estimate for the boundary value problem associated with the eigenvalue problem (\ref{s1.1}).
Here we propose an a posteriori estimator of residual type for (\ref{s1.1}).
Using the $\mathbb{T}$-coercivity
we prove the a posteriori error formulas (see Lemma 3.1),
and give the global upper bound and the local lower bound for the error of eigenfunctions.
In particular, we also prove the reliability of the estimator for transmission eigenvalues.
Numerical experiments indicate that our method is efficient and can reach the optimal order $DoF^{-2m/d}$
by using piecewise polynomials of degree $m$ for real eigenvalues.\\
\indent In this paper, regarding the basic theory of finite element methods, we refer to \cite{babuska1991,boffi2010,brenner2002,ciarlet1991,oden2012}.\\
\indent Throughout this article, for the sake of simplicity, we  write $a \lesssim b$ when $a \leq C b$ and $a \gtrsim b$ when $a \geq C b$ for some positive constant $C$ independent of the finite element mesh.

\section{Finite element approximation and its a priori error estimates}\label{sec:sec2}

\indent In this section, we introduce the finite element approximation and its a priori error estimates for (\ref{s1.1}). We assume that there exists a real number $\gamma>1$ such that
\begin{equation}\label{s2.1}
\xi\cdot A\xi > \gamma|\xi|^2 \quad \forall \xi\in\mathbb{R}^d, n(x) > \gamma, \text{~a.\,e.} ~\mathrm{in}\,~ \Omega.
\end{equation}

\indent Let $H^s(D)$ be the Sobolev space on the domain $D\subseteq \Omega$ with the associated norm $\|\cdot\|_{s,D}$. We omit the subscript $D$ if $D=\Omega$ from the norm notation.

\indent The eigenvalue problem (\ref{s1.1}) can be transformed into the following version:
Find $\lambda = \mathrm{k}^2 +1\in \mathbb{C},(u,w)\in H^1(\Omega)\times H^1(\Omega)$ such that
\begin{equation}\label{s2.2}
\left\{
  \begin{array}{ll}
   -\nabla\cdot(A\nabla w) + n(x)w = \lambda n(x)w ~~&\mathrm{in}\, \Omega, \\
   -\triangle v + v = \lambda v~~&\mathrm{in}\, \Omega, \\
   w-v =0  ~~~~~~~~~~~~~&\mathrm{on} \,\partial\Omega,\\
   \frac{\partial w}{\partial \nu_A}-\frac{\partial v}{\partial \nu}=0~~~~~~~~~~~~~&\mathrm{on}\, \partial\Omega.
  \end{array}
\right.
\end{equation}
\indent Define the function spaces $\mathbb{V}$ and $\mathbb{W}$ as follows:
\begin{align}\label{s2.3}
& \mathbb{V} = \bigg\{\Psi=(\varphi,\psi)\in H^1(\Omega)\times H^1(\Omega)\,|\, \varphi-\psi\in H_0^1(\Omega) \bigg\}, \\\label{s2.4}
& \mathbb{W} = L^2(\Omega)\times L^2(\Omega)
\end{align}
and the associated norm is given by
\begin{equation*}
\|\Psi\|_{\mathbb{V}} = \bigg(\|\varphi\|_1^2 + \|\psi\|_1^2\bigg)^{1/2}
\mathrm{and}\,~
\|\Psi\|_{\mathbb{W}} = \bigg(\|\varphi\|_0^2 + \|\psi\|_0^2\bigg)^{1/2},
\end{equation*}
respectively, where $\mathbf{\Psi}=(\varphi,\psi)$.\\
Then $\mathbb{V}\subset \mathbb{W}$ with a compact imbedding.\\
\indent For $\mathbf{U} = (w,v),\Psi = (\varphi,\psi)\in\mathbb{V},$ we define two sesquilinear forms
\begin{align}\label{s2.5}
& a(\mathbf{U},\Psi) = (A\nabla w,\nabla \varphi)_0 + (n(x)w,\varphi)_0 - (\nabla v,\nabla \psi)_0 -(v,\psi)_0,\\\label{s2.6}
& b(\mathbf{U},\Psi) = (n(x)w,\varphi)_0 - (v,\psi)_0,
\end{align}
where $(w,\varphi)_0 = \int_{\Omega}w\overline{\varphi}\,\mathrm{d}x$.\\
\indent It is clear that
\begin{eqnarray}\label{s2.7}
&&|a( \mathbf{U}, \Psi)|\lesssim \|\mathbf{U}\|_{\mathbb{V}}\|\Psi\|_{\mathbb{V}}\quad\forall  \mathbf{U}, \Psi\in \mathbb{V},\\\label{s2.8}
&&|b( \mathbf{U}, \Psi)|\lesssim \|\mathbf{U}\|_{\mathbb{W}}\|\Psi\|_{\mathbb{W}}\quad\forall  \mathbf{U}, \Psi\in \mathbb{W}.
\end{eqnarray}
\indent Thanks to \cite{dhia2011,xie2017},
the variational form associated with (\ref{s2.2}) is given by:
Find $\lambda\in \mathbb{C}$ and $\mathbf{U}\in \mathbb{V}$, $\mathbf{U}\neq 0$, such that 
\begin{equation}\label{s2.9}
a(\mathbf{U},\Psi) = \lambda b(\mathbf{U},\Psi)\quad \forall \Psi\in\mathbb{V}.
\end{equation}
Then the corresponding adjoint eigenvalue problem is:
Find $\lambda^*\in \mathbb{C}$ and $\mathbf{U}^*\in\mathbb{V}$, $\mathbf{U}^*\neq 0$, such that
\begin{equation}\label{s2.10}
a(\Psi,\mathbf{U}^*) = \bar{\lambda}^* b(\Psi,\mathbf{U}^*)\quad \forall \Psi\in\mathbb{V}.
\end{equation}
\indent Let $\lambda$ be the eigenvalue of (\ref{s2.9}). From page 683 and page 693 in \cite{babuska1991} we know that
 $\mathbf{U}^{j}$ is a generalized eigenfunction of order $j>1$ corresponding to $\lambda$
if and only if
\begin{eqnarray*}
a(\mathbf{U}^{j},\Psi) =\lambda b(\mathbf{U}^{j},\Psi)+\lambda a(\mathbf{U}^{j-1},\Psi)\quad \forall \Psi\in\mathbb{V},
\end{eqnarray*}
where $\mathbf{U}^{j-1}$ is a  generalized eigenfunction of order $j-1$ corresponding to $\lambda$. And
the eigenfunction is a generalized eigenfunction of order $1$.
\begin{theorem}
The number
$\lambda$ is the eigenvalue of (\ref{s2.9}) and $\mathbf{U}^{j}$ is a generalized eigenfunction of order $j\geq 1$  corresponding to $\lambda$
if and only if $\lambda$ is the eigenvalue of (\ref{s2.10}) and $\mathbf{U}^{j}$ is a generalized eigenfunction of order $j\geq 1$ corresponding to $\lambda$.
\end{theorem}
\begin{proof}
The proof is completed by induction.
Let $\lambda$ be the eigenvalue of (\ref{s2.9}) and $\mathbf{U}$ be a generalized eigenfunction of order $1$
corresponding to $\lambda$.
Since $a(\cdot,\cdot)$ and $b(\cdot,\cdot)$ are both self-adjoint, we get
\begin{eqnarray*}
a(\Psi, \mathbf{U})=\overline{a(\mathbf{U},\Psi)} =\overline{ \lambda b(\mathbf{U},\Psi)}
=\overline{ \lambda} b(\Psi, \mathbf{U})\quad \forall \Psi\in\mathbb{V}.
\end{eqnarray*}
This shows $\lambda$ is the eigenvalue of (\ref{s2.10}) and $\mathbf{U}$ is a generalized eigenfunction of order $1$
corresponding to $\lambda$.\\
\indent Suppose that the conclusion 2.1 holds for $j-1$.
Let $\mathbf{U}^{j}$ be a generalized eigenfunction of (\ref{s2.9}) of order $j$
corresponding $\lambda$, then
\begin{eqnarray*}
a(\mathbf{U}^{j},\Psi)=\lambda b(\mathbf{U}^{j},\Psi)+a(\mathbf{U}^{j-1},\Psi)\quad\forall \Psi\in\mathbb{V}.
\end{eqnarray*}
Again from the fact that $a(\cdot,\cdot)$ and $b(\cdot,\cdot)$ are self-adjoint, we deduce
\begin{eqnarray*}
&&a(\Psi, \mathbf{U}^{j})= \overline{a(\mathbf{U}^{j},\Psi)}=\overline{ \lambda b(\mathbf{U}^{j},\Psi)+a(\mathbf{U}^{j-1},\Psi)}\\
&&~~~=\overline{ \lambda} b(\Psi,\mathbf{U}^{j})+a(\Psi, \mathbf{U}^{j-1})\quad\forall \Psi\in\mathbb{V}.
\end{eqnarray*}
This indicates $\lambda$ is the eigenvalue of (\ref{s2.10}) and $\mathbf{U}^{j}$ is a generalized eigenfunction of order $j$ corresponding to  $\lambda$.\\
\indent Conversely, we can prove the other part of the theorem is valid.
\end{proof}
\begin{remark}
When $\lambda$ is a complex eigenvalue and $\mathbf{U}$ is an eigenfunction corresponding to $\lambda$, we have
$$\lambda b(\mathbf{U},\mathbf{U}) = a(\mathbf{U},\mathbf{U}).$$
Note that $b(\mathbf{U},\mathbf{U})$ and $a(\mathbf{U},\mathbf{U})$
are both real numbers, then there must hold $b(\mathbf{U},\mathbf{U}) = 0$ and $a(\mathbf{U},\mathbf{U})=0$.
\end{remark}
\indent In order to analyze the eigenvalue problem (\ref{s2.9}), we introduce the  $\mathbb{T}$-coercivity for the sesquilinear  form $a(\cdot,\cdot)$.
We use $\mathbb{T}$ to denote an isomorphic operator from $\mathbb{V}$ to $\mathbb{V}$ which is defined as follows:
\begin{equation}\label{s2.11}
\mathbb{T}\Psi = (\varphi,2\varphi-\psi) \quad \forall \Psi=(\varphi,\psi)\in\mathbb{V}.
\end{equation}
\indent Thanks to \cite{dhia2011,xie2017}, we know that $a(\cdot,\cdot)$ is $\mathbb{T}$-coercive:
\begin{equation}\label{s2.12}
| a(\Psi,\mathbb{T}\Psi)|\gtrsim \|\Psi\|_{\mathbb{V}}^{2} \quad \forall \Psi=(\varphi,\psi)\in\mathbb{V}.
\end{equation}
In fact,
\begin{eqnarray*}
&&| a(\Psi,\mathbb{T}\Psi)|=|(A\nabla\varphi,\nabla\varphi )_{0}+(n(x)\varphi,\varphi )_{0}
-(\nabla\mathbf{\psi}, \nabla(2\varphi-\psi))_{0}-(\mathbf{\psi}, 2\varphi-\psi)_{0}|\nonumber\\
&&~~~\geq\gamma\|\varphi\|_{1}^{2}+\|\psi\|_{1}^{2}-2|(\nabla\psi,\nabla\varphi)_{0}|
-2|(\psi,\varphi)_{0}|\nonumber\\
&&~~~\geq (\gamma-\delta^{-1})\|\varphi\|_{1}^{2}+(1-\delta)\|\psi\|_{1}^{2},
\end{eqnarray*}
thus we can choose $\delta\in (1/\gamma,1)$ such that (\ref{s2.12}) holds.\\
\indent Theorem 1 in \cite{ciarlet2012} and Theorem 1 in \cite{xie2017} show that the $\mathbb{T}$-coercivity of the form $a(\cdot,\cdot)$ is equivalent to the inf-sup condition of the form $a(\cdot,\cdot)$, namely, there hold
\begin{align}
\label{s2.13}
& \inf_{0\neq\Phi\in\mathbb{V}}\sup_{0\neq\Psi\in\mathbb{V}}\frac{|a(\Phi,\Psi)|}{\|\Phi\|_{\mathbb{V}}\|\Psi\|_{\mathbb{V}}}
\geq\mu_a, \\
\label{s2.14}
& \inf_{0\neq\Phi\in\mathbb{V}}\sup_{0\neq\Psi\in\mathbb{V}}\frac{|a(\Psi,\Phi)|}{\|\Psi\|_{\mathbb{V}}\|\Phi\|_{\mathbb{V}}}
\geq\mu_a,
\end{align}
for some positive constant $\mu_a$.\\
\indent Thanks to Section 8 in \cite{babuska1991} and Chapter 5 in \cite{babuska1973} we know that there are the solution operators
$\mathbb{K}:\mathbb{W}\to\mathbb{V}$ and $\mathbb{K}_{*}:\mathbb{W}\to\mathbb{V}$
satisfying
\begin{eqnarray}\label{s2.15}
&&a(\mathbb{K}\mathbf{F},\Psi) = b(\mathbf{F},\Psi) \quad \forall \Psi\in\mathbb{V},\\\label{s2.16}
&&a(\Psi, \mathbb{K}_{*}\mathbf{F}) = b(\Psi,\mathbf{F}) \quad \forall \Psi\in\mathbb{V},
\end{eqnarray}
and it is valid that
\begin{eqnarray}\label{s2.17}
&&\|\mathbb{K}\mathbf{F}\|_{\mathbb{V}} \lesssim  \|\mathbf{F}\|_{\mathbb{W}}\quad \forall \mathbf{F}\in\mathbb{W},\\\label{s2.18}
&&\|\mathbb{K}_{*}\mathbf{F}\|_{\mathbb{V}} \lesssim  \|\mathbf{F}\|_{\mathbb{W}}\quad \forall \mathbf{F}\in\mathbb{W}.
\end{eqnarray}
Thus, since $\mathbb{V}\subset \mathbb{W}$ with a compact imbedding,
$\mathbb{K}:\mathbb{W}\to \mathbb{W} $ is compact and $\mathbb{K}:\mathbb{V}\to \mathbb{V} $ is compact.\\
\indent From \cite{babuska1991}, we know that (\ref{s2.9}) and (\ref{s2.10}) have the equivalent operator form
\begin{equation}\label{s2.19}
\lambda\mathbb{K}\mathbf{U}= \mathbf{U}~~\mathrm{and}~~\lambda^{*}\mathbb{K}_{*}\mathbf{U}= \mathbf{U},\end{equation}
respectively.\\
\indent Let $\{\mathcal{T}_h\}$ be a family of regular triangulation of $\Omega$ with the mesh diameter $h$,
let $S^{h}\subset H^{1}(\Omega)$ be the conforming Lagrange finite element space on $\mathcal{T}_h$ which consists of piecewise polynomials of degree $m$.
Denote
\begin{equation*}
\mathbb{V}_h=\{(w_{h},v_{h})\in S^{h}\times S^{h}\,|\, w_{h}=v_{h}~\text{on}~\partial\Omega\}.
\end{equation*}
Then
 $\mathbb{V}_h\subset \mathbb{V}$.\\
Thus we also have the following discrete inf-sup conditions:
\begin{equation}\label{s2.20}
\|\Phi_h\|_{\mathbb{V}}
\lesssim
\sup_{\Psi_h\in\mathbb{V}_h}\frac{|a(\Phi_h,\Psi_h)|}{\|\Psi_h\|_{\mathbb{V}}}
\quad \mathrm{and}\quad
\|\Phi_h\|_{\mathbb{V}}
\lesssim
\sup_{\Psi_h\in\mathbb{V}_h}\frac{|a(\Psi_h,\Phi_h)|}{\|\Psi_h\|_{\mathbb{V}}}.
\end{equation}
\indent The finite element approximation of (\ref{s2.9}) is given by: Find $\lambda_h \in \mathbb{C}$ and $\mathbf{U}_h\in\mathbb{V}_{h}$, $\mathbf{U}_h\neq 0$, such that
\begin{equation}\label{s2.21}
a(\mathbf{U}_h,\Psi) = \lambda_hb(\mathbf{U}_h,\Psi) \quad \forall  \Psi\in\mathbb{V}_h.
\end{equation}
The finite element approximation of (\ref{s2.10}) is given by: Find $\lambda_h^{*} \in \mathbb{C}$ and $\mathbf{U}_h^{*}\in\mathbb{V}_{h}$, $\mathbf{U}_h^{*}\neq 0$, such that
\begin{equation}\label{s2.22}
a(\Psi,\mathbf{U}_h^{*}) = \lambda_h^{*}b(\Psi,\mathbf{U}_h^{*}) \quad \forall  \Psi\in\mathbb{V}_h.
\end{equation}
It is clear that Theorem 2.1 also holds for (\ref{s2.21}) and (\ref{s2.22}).\\
\indent Since (\ref{s2.20}) holds, there are
 the solution operators $\mathbb{K}_{h}:\mathbb{W}\to\mathbb{V}_{h}$ and $\mathbb{K}_{h*}:\mathbb{W}\to\mathbb{V}_{h}$ satisfying
\begin{eqnarray}\label{s2.23}
&&a(\mathbb{K}_{h}\mathbf{F},\Psi) = b(\mathbf{F},\Psi)\quad \forall \Psi\in\mathbb{V}_{h},\\\label{s2.24}
&&a(\Psi, \mathbb{K}_{h*}\mathbf{F}) = b(\Psi,\mathbf{F})\quad \forall \Psi\in\mathbb{V}_{h},
\end{eqnarray}
and it is valid that
\begin{eqnarray}\label{s2.25}
&&\|\mathbb{K}_{h}\mathbf{F}\|_{\mathbb{V}} \lesssim  \|\mathbf{F}\|_{\mathbb{W}}\quad \forall \mathbf{F}\in\mathbb{W},\\\label{s2.26}
&&\|\mathbb{K}_{h*}\mathbf{F}\|_{\mathbb{V}} \lesssim  \|\mathbf{F}\|_{\mathbb{W}}\quad \forall \mathbf{F}\in\mathbb{W}.
\end{eqnarray}
\indent From \cite{babuska1991}, we know that (\ref{s2.21}) and (\ref{s2.22}) have the equivalent operator form
\begin{equation}\label{s2.27}
\lambda_h\mathbb{K}_h\mathbf{U}_h= \mathbf{U}_h~~\mathrm{and}~~\lambda_h^{*}\mathbb{K}_h^{*}\mathbf{U}_h=\mathbf{U}_h,
\end{equation}
respectively.\\
\indent Let $P_h: \mathbb{V}\to\mathbb{V}_h$ and $P_h^*: \mathbb{V}\to\mathbb{V}_h$ be the projection operators defined by
\begin{align}\label{s2.28}
a(\Phi - P_h\Phi,\Psi_h)    &= 0\quad \forall\, \Psi_h\in\mathbb{V}_h, \\\label{s2.29}
a(\Psi_h,\Phi-P_h^*\Phi) &= 0\quad \forall\, \Psi_h\in\mathbb{V}_h.
\end{align}
Then $\mathbb{K}_{h}=P_{h}\mathbb{K}$ and $\mathbb{K}_{h*}=P_{h}^{*}\mathbb{K}^{*}$.
\begin{lemma}
Assume that (\ref{s2.1}) is valid, then
\begin{eqnarray*}
\|\mathbb{K}_{h}-\mathbb{K}\|_{\mathbb{V}}\to 0~as~h\to~0~~and~~\|\mathbb{K}_{h*}-\mathbb{K}_{*}\|_{\mathbb{V}}\to 0~as~h\to~0.
\end{eqnarray*}
\end{lemma}
\begin{proof}
From $\mathbb{T}$-coercivity of $a(\cdot,\cdot)$ we deduce that
\begin{eqnarray*}
&&\|\Psi-P_{h}\Psi\|_{\mathbb{V}}^{2}\lesssim a(\Psi-P_{h}\Psi, \mathbb{T}(\Psi-P_{h}\Psi))\\
&&~~~=Ca(\Psi-P_{h}\Psi, \mathbb{T}(\Psi-\Phi_{h}))\\
&&~~~\lesssim\|\Psi-P_{h}\Psi\|_{\mathbb{V}}\|\Psi-\Phi_{h}\|_{\mathbb{V}} \quad\forall \Phi_{h}\in \mathbb{V}_{h},
\end{eqnarray*}
thus, from the interpolation estimate we obtain
\begin{eqnarray}\label{s2.30}
\|\Psi-P_{h}\Psi\|_{\mathbb{V}}\lesssim\inf\limits_{\Phi_{h}\in\mathbb{V}_{h}}\|\Psi-\Phi_{h}\|_{\mathbb{V}}\to 0~~(h\to 0)\quad\forall \Psi\in \mathbb{V}.
\end{eqnarray}
\noindent Since $\mathbb{K}$ is compact, $\mathbb{K}_{h}=P_{h}\mathbb{K}$ converges to $\mathbb{K}$ in norm $||\cdot||_{\mathbb{V}}$ as $h\to 0$.
Similarly, we have $\mathbb{K}_{h*}=P_{h}^{*}\mathbb{K}_{*}$ converges to $\mathbb{K}_{*}$ in norm $||\cdot||_{\mathbb{V}}$ as $h\to 0$.
\end{proof}
\indent Suppose that $\lambda$ and $\lambda_h$ are the $k$th eigenvalue of (\ref{s2.9}) and (\ref{s2.21}), respectively. Let $q$ and $\alpha$ be the algebraic multiplicity and the ascent of $\lambda$, respectively, $\lambda = \lambda_k = \lambda_{k+1} = \cdots = \lambda_{k+q-1}$. Let $M(\lambda)$ be the generalized eigenfunction space of (\ref{s2.9}) corresponding to $\lambda$ and $M_h(\lambda)$ be the direct sum of the generalized eigenfunctions corresponding to all eigenvalues of (\ref{s2.21}) that converge to $\lambda$.
And let $M^{*}(\lambda)$ be the generalized eigenfunction space of (\ref{s2.10}) corresponding to $\lambda^{*}$.
Let $\hat{M}(\lambda) = \{\Psi|\Psi\in M(\lambda), \|\Psi\|_{\mathbb{V}}=1\}$
and $\hat{M}^*(\lambda) = \{\Psi|\Psi\in M^*(\lambda), \|\Psi\|_{\mathbb{V}}=1\}$.
From Theorem 2.1 we know $\lambda^{*}=\lambda$, $M^{*}(\lambda)=M(\lambda)$ and $\hat{M}^*(\lambda)=\hat{M}(\lambda)$.\\
\indent Denote
\begin{equation*}
\delta_h(\lambda) = \sup_{\mathbf{U}\in \hat{M}(\lambda)}\inf_{\Psi_h\in\mathbb{V}_h} \|\mathbf{U}-\Psi_h\|_{\mathbb{V}},\,\delta_h^*(\lambda) = \sup_{\mathbf{U}\in \hat{M}^*(\lambda)}\inf_{\Psi_h\in\mathbb{V}_h} \|\mathbf{U}-\Psi_h\|_{\mathbb{V}}.
\end{equation*}
Then $\delta_h(\lambda) =\delta_h^{*}(\lambda)$.\\
\indent For two closed subspaces $A$ and $B$ of $\mathbb{V}$, we denote
\[
\hat{\Theta}(A,B) = \sup_{\Phi\in A,\|\Phi\|_{\mathbb{V}}=1}\inf_{\Psi\in B}\|\Phi-\Psi\|_{\mathbb{V}},\,
\]
and define the gap between $A$ and $B$ in norm $\|\cdot\|_{\mathbb{V}}$ as
\begin{equation}\label{s2.31}
\Theta(A,B) = \max\{\hat{\Theta}(A,B),\hat{\Theta}(B,A)\}.
\end{equation}
Similarly, we can define the gap $\Theta_{\mathbb{W}}(A,B)$ between two closed subspaces $A$ and $B$ of $\mathbb{W}$
in the sense of norm $\|\cdot\|_{\mathbb{W}}$.\\
\indent Since (\ref{s2.7}), (\ref{s2.8}), (\ref{s2.13}) and (\ref{s2.14}) hold and $\mathbb{V}\subset \mathbb{W}$ with a compact imbedding,
thanks to Theorems 8.1-8.4 in \cite{babuska1991} we obtain the following result.
\begin{theorem}
Assume that $\lambda$ is the $k$th eigenvalue of (\ref{s2.9}), $\hat{\lambda}_h=\frac{1}{q}\sum\limits_{j=k}^{k+q-1}\lambda_{j,h}$ denotes the arithmetic mean of $q$ discrete eigenvalues of (\ref{s2.21}) that converge to $\lambda$, and $h$ is small enough. Then
\begin{align}\label{s2.32}
& \Theta(M(\lambda),M_h(\lambda))\lesssim \delta_h(\lambda),\\\label{s2.33}
&|\lambda - \hat{\lambda}_h|\lesssim \delta_h(\lambda)^{2},\\\label{s2.34}
&|\lambda - \lambda_{j,h}|\lesssim \delta_h(\lambda)^{\frac{2}{\alpha}},~~~j=k,k+1,\cdots,k+q-1.
\end{align}\label{s2.35}
Assume $\mathbf{U}_h\in \mathrm{ker}(\mu_h-\mathbb{K}_h)^\rho$ with $\|\mathbf{U}_h\|_{\mathbb{V}}=1$ for some positive integer $\rho\leq\alpha$. Then for any integer $l$ with $\rho\leq l\leq\alpha$, there exists a generalized eigenfunction $\mathbf{U}$ of (\ref{s2.21}) such that
$(\mu-\mathbb{K})^{l}\mathbf{U}=0$ and
\begin{equation}
\|\mathbf{U}_{h} - \mathbf{U}\|_{\mathbb{V}}\lesssim \delta_h(\lambda)^{(l-\rho+1)/\alpha}.~~~~~~~~~~~~~~~~~~~~~~~~~~
\end{equation}
\end{theorem}
\indent The above Theorem 2.2  was first proved by Xie and Wu for the transmission eigenvalue problem (\ref{s1.1}) (see Theorem 2 in \cite{xie2017}).
\begin{theorem}
Under the conditions of Theorem 2.2, there hold
\begin{eqnarray}\label{s2.36}
&&\Theta_{\mathbb{W}}(M(\lambda),M_{h}(\lambda))\lesssim \|\mathbb{K}_{*}-P_{h}^{*}\mathbb{K}_{*}\|_{\mathbb{V}}^{\frac{1}{2}} \delta_{h}(\lambda),\\\label{s2.37}
&&\|\mathbf{U}_{h}-\mathbf{U}\|_{\mathbb{W
}}\lesssim  (\|\mathbb{K}_{*}-P_{h}^{*}\mathbb{K}_{*}\|_{\mathbb{V}}^{\frac{1}{2}}\delta_{h}(\lambda))^{(l-\rho+1)/\alpha};
\end{eqnarray}
further let $\alpha=1$, then
\begin{eqnarray}\label{s2.38}
\|\mathbf{U}_{h}-\mathbf{U}\|_{\mathbb{W
}}\lesssim  \|\mathbb{K}_{*}-P_{h}^{*}\mathbb{K}_{*}\|_{\mathbb{V}}^{\frac{1}{2}}\|\mathbf{U}_{h}-\mathbf{U}\|_{\mathbb{V
}}.
\end{eqnarray}
\end{theorem}
\begin{proof}
From (\ref{s2.6}), (\ref{s2.1}) and Young's inequality, we deduce
 \begin{eqnarray*}
 &&|b(\Psi, \mathbb{T}\Psi)|
 =|(n(x)\varphi,\varphi )_{0}
-(\mathbf{\varphi}, 2\psi-\varphi)_{0}|\nonumber\\
&&~~~\geq\gamma\|\varphi\|_{0}^{2} +\|\psi\|_{0}^{2}
-2|(\psi,\varphi)_{0}|\nonumber\\
&&~~~\geq (\gamma-\delta^{-1})\|\varphi\|_{0}^{2}+(1-\delta)\|\psi\|_{0}^{2}.
 \end{eqnarray*}
Then we can choose $\delta\in(1/\gamma, 1)$ and $C_{b}=\min\{(\gamma-\delta^{-1}), (1-\delta)\}>0$ such that
\begin{eqnarray}\label{s2.39}
|b(\Psi, \mathbb{T}\Psi)|\geq C_{b}\|\Psi\|_{\mathbb{W}}^{2}\quad\forall \Psi\in \mathbb{W},
\end{eqnarray}
which means $b(\cdot,\cdot)$ is  $\mathbb{T}$-coercive on $\mathbb{W}$.\\
\indent For any $\mathbf{F}\in \mathbb{W}$,
from (\ref{s2.39}), (\ref{s2.16}) and (\ref{s2.29}) we deduce
\begin{eqnarray*}
&&C_{b}\|\mathbb{K}\mathbf{F}-P_{h}\mathbb{K}\mathbf{F}\|_{\mathbb{W}}^{2}\leq | b(\mathbb{K}\mathbf{F}-P_{h}\mathbb{K}\mathbf{f}, \mathbb{T}(\mathbb{K}\mathbf{F}-P_{h}\mathbb{K}\mathbf{F}))|\nonumber\\
&&~~~=|a(\mathbb{K}\mathbf{F}-P_{h}\mathbb{K}\mathbf{F}, \mathbb{K}_{*}\mathbb{T}(\mathbb{K}\mathbf{F}-P_{h}\mathbb{K}\mathbf{F}))|\nonumber\\
&&~~~=|a(\mathbb{K}\mathbf{F}-P_{h}\mathbb{K}\mathbf{F}, \mathbb{K}_{*}\mathbb{T}(\mathbb{K}\mathbf{F}-P_{h}\mathbb{K}\mathbf{F})-P_{h}^{*}\mathbb{K}_{*}\mathbb{T}(\mathbb{K}\mathbf{F}-P_{h}\mathbb{K}\mathbf{F}))|\nonumber\\
&&~~~\lesssim \|\mathbb{K}_{*}-P_{h}^{*}\mathbb{K}_{*}\|_{\mathbb{V}}\|\mathbb{K}\mathbf{F}-P_{h}\mathbb{K}\mathbf{F}\|_{\mathbb{V}}^{2},
\end{eqnarray*}
then,
\begin{eqnarray}\label{s2.40}
\|\mathbb{K}\mathbf{F}-P_{h}\mathbb{K}\mathbf{F}\|_{\mathbb{W}}\lesssim \|\mathbb{K}_{*}-P_{h}^{*}\mathbb{K}_{*}\|_{\mathbb{V}}^{\frac{1}{2}}\|\mathbb{K}\mathbf{F}-P_{h}\mathbb{K}\mathbf{F}\|_{\mathbb{V}}.
\end{eqnarray}
\indent From (\ref{s2.40}), (\ref{s2.17}) and (\ref{s2.25}) we deduce
\begin{eqnarray}\label{s2.41}
&&\|\mathbb{K}-P_{h}\mathbb{K}\|_{\mathbb{W}}=\sup\limits_{\mathbf{F}\in \mathbb{W}, \|\mathbf{F}\|_{\mathbb{W}}=1}
\|\mathbb{K}\mathbf{F}-P_{h}\mathbb{K}\mathbf{F}\|_{\mathbb{W}}\nonumber\\
&&~~~\lesssim \|\mathbb{K}_{*}-P_{h}^{*}\mathbb{K}_{*}\|_{\mathbb{V}}^{\frac{1}{2}}\sup\limits_{\mathbf{F}\in \mathbb{W}, \|\mathbf{F}\|_{\mathbb{W}}=1}\|\mathbb{K}\mathbf{F}-P_{h}\mathbb{K}\mathbf{F}\|_{\mathbb{V}}\nonumber\\
&&~~~\lesssim \|\mathbb{K}_{*}-P_{h}^{*}\mathbb{K}_{*}\|_{\mathbb{V}}^{\frac{1}{2}}\to 0~~as~~h\to 0.
\end{eqnarray}
Thus from Theorem 7.1
in \cite{babuska1991}, (\ref{s2.17}), (\ref{s2.40}) and (\ref{s2.30}) we derive (\ref{s2.36}) as follows:
\begin{eqnarray*}
&&{\Theta}_{\mathbb{W}}(M(\lambda),M_{h}(\lambda))\lesssim \|(\mathbb{K}-P_{h}\mathbb{K})|_{M(\lambda)}\|_{\mathbb{W}}\\
&&~~~=\sup\limits_{\Psi\in M(\lambda)}\frac{\|(\mathbb{K}-P_{h}\mathbb{K})\Psi\|_{\mathbb{W}}}{\|\Psi\|_{\mathbb{\mathbb{W}}}}\lesssim
\sup\limits_{\Psi\in M(\lambda)}\frac{\|(\mathbb{K}-P_{h}\mathbb{K})\Psi\|_{\mathbb{W}}}{\|\mathbb{K}\Psi\|_{\mathbb{\mathbb{V}}}}\nonumber\\
&&~~~\lesssim\|\mathbb{K}_{*}-P_{h}^{*}\mathbb{K}_{*}\|_{\mathbb{V}}^{\frac{1}{2}}
\sup\limits_{\Psi\in M(\lambda)}\frac{\|(\mathbb{I}-P_{h})\mathbb{K}\Psi\|_{\mathbb{V}}}{\|\mathbb{K}\Psi\|_{\mathbb{\mathbb{V}}}}\nonumber\\
&&~~~\lesssim \|\mathbb{K}_{*}-P_{h}^{*}\mathbb{K}_{*}\|_{\mathbb{V}}^{\frac{1}{2}}\delta_{h}(\lambda),
\end{eqnarray*}
where $\mathbb{I}$ is an identity operator. \\ 
\indent From Theorem 7.4 in \cite{babuska1991}
 and (\ref{s2.40}) we obtain (\ref{s2.37}).\\
\indent When $\alpha=1$, by the spectral approximation theory we have
\begin{eqnarray*}
\|\mathbb{U}_{h}-\mathbb{U}\|_{\mathbb{W
}}\lesssim  \|(\mathbb{K}-P_{h}\mathbb{K})\mathbb{U}\|_{\mathbb{W}}.
\end{eqnarray*}
Thus, from (\ref{s2.40}) and (\ref{s2.30}) we get (\ref{s2.38}).
The proof is completed.
\end{proof}

\begin{remark}
We define another function spaces as follows.
$$ \tilde{\mathbb{V}} = \bigg\{\Psi=(\psi,\varphi)\in \mathbb{V}:\, \int_{\Omega}(n\psi-\varphi)dx = 0 \bigg\},$$
\begin{equation*}
\tilde{\mathbb{V}}_h=\left\{(\psi_h,\varphi_h)\in \mathbb{V}_h: \int_{\Omega}(n\psi_h-\varphi_h)dx = 0 \right\}.
\end{equation*}

From \cite{dhia2011}, let $\lambda=\mathrm{k}^2$, then the variational formulation of (\ref{s1.1}) is as follows: Find $\lambda\in \mathbb{C}$ and $U \in \tilde{\mathbb{V}}$, $U \neq 0$, such that
\begin{equation}\label{s2.42}
\tilde{a}(U ,\Psi) = \lambda b(U ,\Psi)\quad \forall \Psi\in\tilde{\mathbb{V}},
\end{equation}
where
$
\tilde{a}(U ,\Psi) = (A\nabla u,\nabla \psi) - (\nabla v,\nabla \varphi).
$
It is clear that the eigenpair of (\ref{s2.42}) is the same with the one of its dual problem.

\indent The finite element approximation of  (\ref{s2.42}) is to find $\lambda_{h}=\mathrm{k}^2_{h}\in \mathbb{C}$, $U_h \in \tilde{\mathbb{V}}_{h}$, $U_{h} \neq 0$, such that
\begin{align}\label{s2.43}
\tilde{a}(U_h ,\Psi_h )=\lambda_{h} b(U_h ,\Psi_h )\quad \forall \Psi_h \in \tilde{\mathbb{V}}_{h}.
\end{align}

Referring to \cite{yang2022}, we can prove that $(\lambda, \mathbf{U})$
 is the eigenpair of (\ref{s1.1}) and $\lambda\not=0$ if and only if $(\lambda, \mathbf{U})$ is the eigenpair of (\ref{s2.42}) and that  $(\lambda_{h}, U_h)\in \mathbb{C}\times \mathbb{V}_{h}$, $\lambda_{h}\neq 0$, satisfying
\begin{align}\label{s2.44}
\tilde{a}(U_h ,\Psi_h )=\lambda_{h} b(U_h ,\Psi_h )\quad \forall \Psi_h \in \mathbb{V}_{h},
\end{align}
 if and only if $(\lambda_h, \mathbf{U}_h)$ is the eigenpair of (\ref{s2.43}).
Moreover, if $\int_{\Omega} (n-1)d x\not= 0$, then Theorem 2.2 holds. Further, assume that $n>n_{\mathrm{min}}>1$ or $n<n_{\max}<1$, then Theorem 2.3 holds.
\end{remark}

\section{A posteriori error estimation}\label{sec:sec3}
\subsection{The a posteriori error estimators and their reliability for eigenpair}
\indent Let $\mathcal{E}_h$ denote the set of all faces (or edges when $d=2$) in the mesh. We decompose $\mathcal{E}_h$
into disjoint sets $\mathcal{E}_h^i$ and $\mathcal{E}_h^b$ which consists of inner faces and those on the boundary, respectively.
Let $\kappa$ denote the element and $e$ the face of element in the mesh.
For $e\in\mathcal{E}_h^i$ which is the common side of elements ${\kappa}^+$ and ${\kappa}^-$ with unit outward normal $\nu^+$ and $\nu^-$, respectively, we fix $\nu = \nu^-$ on the boundary $\mathcal{E}_h^b$. Let $\mathbf{U}_h=(w_h,v_h)$ be the solution of (\ref{s2.21}), define the jumps
$[\![ A\nabla w_h]\!]$ and $[\![ \nabla v_h]\!]$ across $e\in \mathcal{E}_h^i$
 as follows:
\begin{align*}
[\![ A\nabla w_h]\!] &= A\nabla w_h|{\kappa}^+\cdot\nu^+ + A\nabla w_h|{\kappa}^-\cdot\nu^-, \\
[\![ \nabla v_h]\!] &= \nabla v_h|{\kappa}^+\cdot\nu^+ + \nabla v_h|{\kappa}^-\cdot\nu^-.
\end{align*}
\indent For any ${\kappa}\in\mathcal{T}_h$, we define the local error estimator $\eta_{\kappa}$ as
\begin{equation}\label{s3.1}
\eta_{\kappa}(\mathbf{U}_h) =\bigg( h^2_{\kappa}R_{\kappa}(\mathbf{U}_h) + \frac12\sum_{e\subset (\partial {\kappa}\cap \mathcal{E}_h^i)}h_e J_{e}^i(\mathbf{U}_h) + \sum_{e\subset (\partial {\kappa}\cap \mathcal{E}_h^b)}h_eJ_{e}^b(\mathbf{U}_h)\bigg)^{1/2},
\end{equation}
where
\begin{align*}
& \, R_{\kappa}(\mathbf{U}_h) = \|\nabla\cdot(A\nabla w_h) + (\lambda-1)n(x)w_h\|_{0,{\kappa}}^2 +\|\Delta v_h + (\lambda-1)v_h\|_{0,{\kappa}}^2,\\
& \, J_{e}^i(\mathbf{U}_h) = \|[\![ A\nabla w_h]\!]\|^2_{0,e}+\|[\![ \nabla v_h]\!]\|^2_{0,e}\quad \mathrm{for}\,e\in\mathcal{E}_h^i,\\
& \, J_{e}^b(\mathbf{U}_h) = \|(A\nabla w_h-\nabla v_h)\cdot\nu\|^2_{0,e} \quad \mathrm{for}\,e\in\mathcal{E}_h^b.
\end{align*}
The global error estimator is then given by
\begin{equation}\label{s3.2}
\eta(\mathbf{U}_h) = \bigg(\sum_{{\kappa}\in\mathcal{T}_h}\eta_{\kappa}(\mathbf{U}_h)^{2}\bigg)^{1/2}.
\end{equation}
\indent Let $\Pi_h:\,H^1(\Omega)\to S^h$ be the Scott-Zhang interpolation operator, then, from \cite{scott1990} we know that there hold the local interpolation estimates: for any $\psi\in H^1(\Omega)$,
\begin{align}\label{s3.3}
 & \, \|\psi-\Pi_h\psi\|_{0,{\kappa}} \lesssim h_{{\kappa}}\|\nabla\psi\|_{0,\tilde{\omega}_{\kappa}}\quad\forall {\kappa}\in \mathcal{T}_{h},\\\label{s3.4}
 & \, \|\psi-\Pi_h\psi\|_{0,e} \lesssim h_{e}^{1/2}\|\nabla\psi\|_{0,\tilde{\omega}_e}\quad\forall e\in \mathcal{E}_{h}.
\end{align}
where $\tilde{\omega}_{\kappa}$ and $\tilde{\omega}_e$ denote the set of all elements that share at least a vertex with ${\kappa}$ and $e$, respectively.\\
\indent We extend Theorem 1.5.2 in \cite{shi2013} to the following a posteriori error formulas.
\begin{lemma}
Let $(\lambda,\mathbf{U})$  be the eigenpair of  (\ref{s2.9}) and $(\lambda_{h},\mathbf{U}_h)$ be the eigenpair of (\ref{s2.21}). Then
\begin{eqnarray}\label{s3.5}
&&\|\mathbf{U}-\mathbf{U}_h\|_{\mathbb{V}}
\lesssim
\sup_{0\neq \Psi\in \mathbb{V}} \frac{|b(\lambda_h \mathbf{U}_h,\Psi) - a(\mathbf{U}_h,\Psi)|}{\|\Psi\|_{\mathbb{V}} }
+ \|\lambda \mathbf{U}-\lambda_h \mathbf{U}_h\|_{\mathbb{W}},\\\label{s3.6}
&&\|\mathbf{U}-\mathbf{U}_h\|_{\mathbb{V}}
\gtrsim
\sup_{0\neq \Psi\in \mathbb{V}} \frac{|b(\lambda_h \mathbf{U}_h,\Psi) - a(\mathbf{U}_h,\Psi)|}{\|\Psi\|_{\mathbb{V}} }- \|\lambda \mathbf{U}-\lambda_h \mathbf{U}_h\|_{\mathbb{W}}.
\end{eqnarray}
\end{lemma}
\begin{proof}
Using the inf-sup condition (\ref{s2.13}), we derive
\begin{equation}\label{s3.7}
\mu_a\|\mathbf{U}-\mathbf{U}_h\|_{\mathbb{V}}
\leq \sup_{0\neq \Psi\in \mathbb{V}} \frac{|a(\mathbf{U}-\mathbf{U}_h,\Psi)|}{\|\Psi\|_{\mathbb{V}} }.
\end{equation}
From (\ref{s2.9}) and (\ref{s2.8}), we deduce
\begin{eqnarray}\label{s3.8}
&&|a(\mathbf{U}-\mathbf{U}_h,\Psi)| =| b(\lambda \mathbf{U},\Psi) - a(\mathbf{U}_h,\Psi)|\nonumber\\
&&~~~=|b(\lambda_h \mathbf{U}_h,\Psi) + b(\lambda \mathbf{U}-\lambda_h \mathbf{U}_h,\Psi)- a(\mathbf{U}_h,\Psi)|\nonumber\\
&&~~~\leq |b(\lambda_h \mathbf{U}_h,\Psi)- a(\mathbf{U}_h,\Psi)| + C\|\lambda \mathbf{U}-\lambda_h \mathbf{U}_h\|_{\mathbb{W}}\|\Psi\|_{\mathbb{W}}.
\end{eqnarray}
Substituting (\ref{s3.8}) into (\ref{s3.7}), we obtain (\ref{s3.5}).\\
On the other hand,
\begin{eqnarray*}
&& |b(\lambda_h \mathbf{U}_h,\Psi) - a(\mathbf{U}_h,\Psi)|
=  |b(\lambda_h \mathbf{U}_h - \lambda \mathbf{U},\Psi) + a(\mathbf{U} - \mathbf{U}_h,\Psi)|\\
&&~~~\lesssim \|\lambda_h \mathbf{U}_h - \lambda \mathbf{U}\|_{\mathbb{W}}\|\Psi\|_{\mathbb{W}}+\|\mathbf{U} - \mathbf{U}_h\|_{\mathbb{V}}\|\Psi\|_{\mathbb{V}}.
\end{eqnarray*}
From the above estimate we immediately get (\ref{s3.6}).
\end{proof}
\begin{theorem}
Let $(\lambda,\mathbf{U})$ and $(\lambda_h,\mathbf{U}_h)$ be the eigenpairs  of  (\ref{s2.9}) and (\ref{s2.21}), respectively. Then
\begin{equation}\label{s3.9}
\|\mathbf{U}-\mathbf{U}_h\|_{\mathbb{V}}
\lesssim  \eta(\mathbf{U}_h)
+ \|\lambda \mathbf{U}-\lambda_h \mathbf{U}_h\|_{\mathbb{W}}.
\end{equation}
\end{theorem}
\begin{proof}
An integration by parts elementwise yields for any $\Psi\in\mathbb{V}$,
\begin{equation}\label{s3.10}
\begin{aligned}
&\lambda_hb( \mathbf{U}_h,\Psi) - a(\mathbf{U}_h,\Psi)\\
=& (\lambda_h n(x)w_h,\varphi)_0 - (\lambda_h v_h,\psi)_0 - (A\nabla w_h,\nabla\varphi)_0 - (n(x)w_h,\varphi)_0 + (\nabla v_h,\nabla\psi)_0 + (v_h,\psi)_0\\
=& \sum_{{\kappa}\in \mathcal{T}_{h}}\int_{{\kappa}}( \nabla\cdot(A\nabla w_h) + (\lambda_h-1)n(x)w_h   )\bar{\varphi}-( \Delta v_h + (\lambda_h-1) v_h  )\bar{\psi}\,\mathrm{d}x \\
&-\sum_{{\kappa}\in \mathcal{T}_{h}}\int_{\partial {\kappa}}( A\nabla w_h\cdot\nu )\bar{\varphi}
-( \nabla v_h \cdot \nu )\bar{\psi}\,\mathrm{d}s\\
=& \sum_{{\kappa}\in \mathcal{T}_{h}}\int_{{\kappa}}( \nabla\cdot(A\nabla w_h) + (\lambda_h-1)n(x)w_h   )\bar{\varphi}-( \Delta v_h + (\lambda_h-1) v_h  )\bar{\psi}\,\mathrm{d}x \\
&
-\frac12\sum_{\kappa\in \mathcal{T}_{h}}\sum_{e\subset (\partial {\kappa}\cap \mathcal{E}_h^i)}\int_{e}( [\![ A\nabla w_h]\!] )\bar{\varphi}-( [\![ \nabla v_h]\!] )\bar{\psi}\,\mathrm{d}s \\
&
-\sum_{\kappa\in \mathcal{T}_{h}}\sum_{e\subset (\partial {\kappa}\cap \mathcal{E}_h^b)}\int_{e}( ( A\nabla w_h-\nabla v_h)\cdot\nu )\bar{\varphi} \,\mathrm{d}s.
\end{aligned}
\end{equation}
Note $\Pi_h\Psi = (\Pi_h\varphi,\Pi_h\psi)$ and $\Psi-\Pi_h\Psi\in \mathbb{V}$, from
(\ref{s2.21}) and (\ref{s3.10}) we deduce
\begin{eqnarray}\label{s3.11}
&&\lambda_hb( \mathbf{U}_h,\Psi) - a(\mathbf{U}_h,\Psi)
=\lambda_hb( \mathbf{U}_h,\Psi-\Pi_h\Psi) - a(\mathbf{U}_h,\Psi-\Pi_h\Psi)\nonumber\\
&&~~~= \sum_{{\kappa}\in \mathcal{T}_{h}}\int_{{\kappa}}( \nabla\cdot(A\nabla w_h) + (\lambda_h-1)n(x)w_h   )(\overline{  \varphi - \Pi_h\varphi  })\nonumber\\
&&~~~-( \Delta v_h + (\lambda_h-1) v_h  )(\overline{\psi - \Pi_h\psi})\,\mathrm{d}x \nonumber\\
&&~~~
-\frac12\sum_{\kappa\in \mathcal{T}_{h}}\sum_{e\subset (\partial {\kappa}\cap \mathcal{E}_h^i)}\int_{e}( [\![ A\nabla w_h]\!] )(\overline{  \varphi - \Pi_h\varphi  })-( [\![ \nabla v_h]\!] )(\overline{\psi - \Pi_h\psi})\,\mathrm{d}s \nonumber\\
&&~~~
-\sum_{\kappa\in \mathcal{T}_{h}}\sum_{e\subset (\partial {\kappa}\cap \mathcal{E}_h^b)}\int_{e}( ( A\nabla w_h-\nabla v_h)\cdot\nu )(\overline{  \varphi - \Pi_h\varphi  }) \,\mathrm{d}s.
\end{eqnarray}
Using the Cauchy-Schwarz inequality and (\ref{s3.11}) yields
\begin{equation}\label{s3.12}
\begin{aligned}
&|\lambda_h b( \mathbf{U}_h,\Psi) - a(\mathbf{U}_h,\Psi)|\\
\lesssim & \sum_{{\kappa}\in \mathcal{T}_{h}}\bigg(\| (\lambda_h-1)n(x)w_h + \nabla\cdot(A\nabla w_h) \|_{0,{\kappa}}\|  \overline{\varphi - \Pi_h\varphi}   \|_{0,{\kappa}} \\
&+\frac12\sum_{e\subset (\partial {\kappa}\cap \mathcal{E}_h^i)}\|  [\![ A\nabla w_h]\!] \|_{0,e}\|  \overline{\varphi - \Pi_h\varphi}  \|_{0,e}
+\sum_{e\subset (\partial {\kappa}\cap \mathcal{E}_h^b)}\|( A\nabla w_h-\nabla v_h)\cdot\nu\|_{0,e}\|  \overline{\varphi - \Pi_h\varphi}  \|_{0,e}\bigg) \\
&+\sum_{{\kappa}\in \mathcal{T}_{h}}\bigg(\|  (\lambda_h-1) v_h + \Delta v_h \|_{0,{\kappa}}\|  \overline{\psi - \Pi_h\psi}  \|_{0,{\kappa}}
+\frac12\sum_{e\subset (\partial {\kappa}\cap \mathcal{E}_h^i)}\|  [\![ \nabla v_h]\!] \|_{0,e}\|  \overline{\psi - \Pi_h\psi}  \|_{0,e}\bigg).\\
\end{aligned}
\end{equation}
From the interpolation estimates (\ref{s3.3}) and (\ref{s3.4})  and inverse estimates, we deduce
\begin{equation}
\begin{aligned}
&|\lambda_hb( \mathbf{U}_h,\Psi) - a(\mathbf{U}_h,\Psi)|\\
\lesssim & \sum_{{\kappa}\in \mathcal{T}_{h}}\bigg(h_{{\kappa}}^2\| (\lambda_h-1)n(x)w_h + \nabla\cdot(A\nabla w_h) \|_{0,{\kappa}}^2 \\
&+\frac12\sum_{e\subset (\partial {\kappa}\cap \mathcal{E}_h^i)}h_{e}\|  [\![ A\nabla w_h]\!] \|_{0,e}^2
+\sum_{e\subset (\partial {\kappa}\cap \mathcal{E}_h^b)}h_{e}\|( A\nabla w_h-\nabla v_h)\cdot\nu\|^2_{0,e}
\bigg)^{1/2}\|\nabla\varphi\|_{0,\omega_{\kappa}}\\
&+\sum_{{\kappa}\in \mathcal{T}_{h}}\bigg(h_{{\kappa}}^2\|  (\lambda_h-1) v_h + \Delta v_h \|_{0,{\kappa}}^2
+\frac12\sum_{e\subset (\partial {\kappa}\cap \mathcal{E}_h^i)}h_{e}\|  [\![ \nabla v_h]\!] \|_{0,e}^2\bigg)^{1/2}\|\nabla\psi\|_{0,\omega_{\kappa}}\nonumber\\
&\lesssim \eta(\mathbf{U}_h)\|\Psi\|_{\mathbb{V}}.
\end{aligned}
\end{equation}
Substituting the above inequality into (\ref{s3.5}), we obtain (\ref{s3.9}). This completes the proof.
\end{proof}

\begin{remark}
A simple calculation shows that
\begin{equation*}
\|\lambda\mathbf{U}-\lambda_{h}\mathbf{U}_{h}\|_{\mathbb{W}}\lesssim |\lambda-\lambda_{h}|+\|\mathbf{U}-\mathbf{U}_{h}\|_{\mathbb{W}}.
\end{equation*}
From Theorems 2.2 and 2.3 we know that $|\lambda-\lambda_{h}|$ and $\|\mathbf{U}-\mathbf{U}_{h}\|_{\mathbb{W}}$ are both
small quantities of higher order compared with $\|\mathbf{U}-\mathbf{U}_{h}\|_{\mathbb{V}}$,
then
$\|\lambda\mathbf{U}-\lambda_{h}\mathbf{U}_{h}\|_{\mathbb{W}}$ is a small quantity of higher order compared with $\|\mathbf{U}-\mathbf{U}_{h}\|_{\mathbb{V}}$.
Hence, Theorem 3.1 shows that when $h$ is small enough, the error estimator $\eta(\mathbf{U}_h)$ is reliable for eigenfunction $\mathbf{U}_{h}$ up to the higher order term $ \|\lambda \mathbf{U}-\lambda_h \mathbf{U}_h\|_{\mathbb{W}}$.
\end{remark}
\indent Referring to \cite{dai2014},
we give the following a posteriori error estimate for  transmission eigenvalues.
\begin{theorem}
Let $\lambda$ and  $\lambda_h$ be the $k$th eigenvalue of  (\ref{s2.9}) and (\ref{s2.21}), respectively.
Assume that
$\alpha=1$, the set of eigenfunctions $\{\mathbf{U}_{j,h}\}_{j=k}^{k+q-1}$ is
 an orthonormal basis of $M_{h}(\lambda)$ satisfying $\|\mathbf{U}_{j,h}\|_{\mathbb{V}}=1$,
and $h$ is small enough, then
\begin{equation}\label{s3.13}
|\lambda_h-\lambda|\lesssim \sum\limits_{j=k}^{k+q-1}
\eta_{h}(\mathbf{U}_{j,h})^{2}.
\end{equation}
\end{theorem}
\begin{proof}
Let $E$ be the spectral projection associated with $\mathbb{K}$ and $\lambda$ (see Section 6 in \cite{babuska1991}).
Then $\mathbf{U}_{j,h}-E\mathbf{U}_{j,h}$ satisfies Theorem 2.2,
and $\{E\mathbf{U}_{j,h}\}_{j=k}^{k+q-1}$ is a basis of $M(\lambda)$.
Thus
\begin{eqnarray}\label{s3.14}
&&\delta_h(\lambda) = \sup\limits_{\mathbf{U}\in \hat{M}(\lambda)}\inf_{\Psi_h\in\mathbb{V}_h} \|\mathbf{U}-\Psi_h\|_{\mathbb{V}}\nonumber\\
&&~~~\lesssim \sum\limits_{j=k}^{k+q-1}\inf_{\Psi_h\in\mathbb{V}_h} \|\frac{E\mathbf{U}_{j,h}}{\|E\mathbf{U}_{j,h}\|_{\mathbb{V}}}-\Psi_h\|_{\mathbb{V}}
\lesssim \sum\limits_{j=k}^{k+q-1}\inf_{\Psi_h\in\mathbb{V}_h} \|\frac{E\mathbf{U}_{j,h}}{\|E\mathbf{U}_{j,h}\|_{\mathbb{V}}}-\frac{\Psi_h}{\|\Psi_h\|_{\mathbb{V}}}\|_{\mathbb{V}}\nonumber\\
&&~~~\lesssim \sum\limits_{j=k}^{k+q-1}\inf_{\Psi_h\in\mathbb{V}_h} \|E\mathbf{U}_{j,h}-\Psi_h\|_{\mathbb{V}}
\lesssim \sum\limits_{j=k}^{k+q-1} \|E\mathbf{U}_{j,h}-\mathbf{U}_{j,h}\|_{\mathbb{V}}\nonumber\\
&&~~~\lesssim \sum\limits_{j=k}^{k+q-1} \eta_{h}(\mathbf{U}_{j,h}).
\end{eqnarray}
Substituting (\ref{s3.14}) into (\ref{s2.34}), we get (\ref{s3.13}) immediately.
\end{proof}
\subsection{The local lower bound of the error for eigenfunction}
In this subsection,
we will use the bubble function techniques developed by Verfürth to prove that
the local error estimator $\eta_{\kappa}$ provides a local lower bound for the error on a neighborhood of $\kappa$.\\
\indent For  $\kappa \in \mathcal{T}_h$, let $\mathfrak{b}_{\kappa}\in H^{1}(\kappa)$ satisfying $\mathfrak{b}_{\kappa}|_{\Omega\backslash\kappa}=0$ be the element bubble function, and
for $e\in \mathcal{E}_h$, let $\mathfrak{b}_{e}\in H^{1}(\omega_{e})$ satisfying $\mathfrak{b}_{e}|_{\Omega\backslash\omega_{e}}=0$ be the face bubble function,
where $\omega_e$ is the union of all elements that share $e$.
Then the following lemma holds (see \cite{verfurth1996,verfurth1998}).
\begin{lemma}
For any $\kappa\in\mathcal{T}_h$ and for any $\phi\in P_m(\kappa)$, there hold
\begin{eqnarray}\label{s3.15}
&&\|\phi\|_{0,\kappa}^{2}
\lesssim\int_{\kappa}(\mathfrak{b}_{\kappa}\phi)\bar{\phi}\,\mathrm{d}x,\\\label{s3.16}
&&\|\mathfrak{b}_{\kappa}\phi\|_{0,\kappa}
\lesssim\|\phi\|_{0,\kappa}.
\end{eqnarray}
For any $e\in\mathcal{E}_h$ and for any $\phi\in P_m|_{e}$, there holds
 \begin{eqnarray}\label{s3.17}
\|\phi \|_{0,e} \lesssim  \|\mathfrak{b}_{e}^{1/2}\phi\|_{0,e},
 \end{eqnarray}
and for each $\mathfrak{b}_{e}\phi|_{e}$, there exists an extension $\psi_{e}$ on $\omega_{e}$ satisfying $\psi_{e}|_{e}=\mathfrak{b}_{e}\phi$, $\psi_{e}|_{ \partial\omega_{e}\backslash\partial\Omega}=0$ and
 \begin{eqnarray}\label{s3.18}
\|\psi_{e}\|_{0,\omega_e}\lesssim
 h_{e}^{1/2}\|\phi\|_{0,e}.
\end{eqnarray}
\end{lemma}
\indent In what follows, for the sake of simplicity, we shall restrict ourselves to the case
that $A(x)$ is a matrix function whose elements are polynomials and $n(x)$ is a polynomial.
 The general case requires only technical modifications.
\begin{theorem}
Let $(\lambda,\mathbf{U})$  be the eigenpair  of  (\ref{s2.9}) and let  $(\lambda_h,\mathbf{U}_h)$ be the eigenpair of (\ref{s2.21}). Then
\begin{equation}\label{s3.19}
\eta_{\kappa}(\mathbf{U}_{h})
\lesssim\bigg(
\|\mathbf{U}-\mathbf{U}_h\|_{1,\omega_{\kappa}}
+ h.o.t.
\bigg)
\end{equation}
with the high order terms
\[
\begin{aligned}
h.o.t. =& \sum_{\kappa\in\omega_{\kappa} }h_{\kappa}\bigg(
\|n(x)((\lambda_h-1)w_h - (\lambda-1)w  )\|_{0,\kappa}
+\|(\lambda_h-1)v_h - (\lambda-1)v \|^2_{0,\kappa}
\bigg),\\
\end{aligned}
\]
where $\omega_{\kappa}$ denotes the union of all elements sharing at least one common face with $\kappa$.
\end{theorem}
\begin{proof}
We now estimate each term of the right hand side of (\ref{s3.1}).\\
(i)~ For $\kappa\in\mathcal{T}_h$, choose $\phi_{\kappa}= \mathfrak{b}_{\kappa}( \nabla\cdot(A\nabla w_h) + (\lambda_h-1)n(x)w_h   )|_{\kappa}$. Since $\mathrm{supp}(\phi_{\kappa})\subset\kappa$, from the first equation in (\ref{s2.2}) and Green's formula we deduce
\[
\begin{aligned}
&\int_{{\kappa}}( \nabla\cdot(A\nabla w_h) + (\lambda_h-1)n(x)w_h)\bar{\phi}_{\kappa}\,\mathrm{d}x\\
=&\int_{{\kappa}}n(x)((\lambda_h-1)w_h - (\lambda-1)w  )\bar{\phi}_{\kappa} \,\mathrm{d}x
+ \int_{{\kappa}}(A\nabla w)\cdot\nabla\bar{\phi}_{\kappa}\,\mathrm{d}x
+ \int_{{\kappa}}( \nabla\cdot(A\nabla w_h)\bar{\phi}_{\kappa} \,\mathrm{d}x \\
\lesssim
& \left(\|n(x)((\lambda_h-1)w_h - (\lambda-1)w  )\|_{0,\kappa}
+h_{\kappa}^{-1}\|A\nabla (w-w_h)\|_{0,\kappa}\right)\|\bar{\phi}_{\kappa}\|_{0,\kappa}.
\end{aligned}
\]
Using (\ref{s3.15}) and (\ref{s3.16}) in Lemma 3.2,
we deduce
\begin{eqnarray}\label{s3.20}
&&~~~h_{\kappa}\|\nabla\cdot(A\nabla w_h) + (\lambda_h-1)n(x)w_h\|_{0,\kappa}\nonumber\\
&&\lesssim h_{\kappa}\|n(x)((\lambda_h-1)w_h - (\lambda-1)w  )\|_{0,\kappa}
+
\|A\nabla (w-w_h)\|_{0,\kappa}.
\end{eqnarray}
Similarly, we can deduce
\begin{equation}\label{s3.21}
h_{\kappa}\|\Delta v_h + (\lambda_h-1)v_h\|_{0,\kappa}
\lesssim h_{\kappa}\|(\lambda_h-1)v_h - (\lambda-1)v \|_{0,\kappa}
+\|\nabla (v-v_h)\|_{0,\kappa}.
\end{equation}
(ii)~ For $e\in\mathcal{E}_h^i$, let $\psi_{e}\in H_{0}^{1}(\omega_e)$ be an extension of $\mathfrak{b}_{e}([\![ A\nabla w_h]\!] )|_e$ satisfying
(\ref{s3.17}) and (\ref{s3.18}).
Then
\[
\begin{aligned}
&\int_{{e}} ([\![ A\nabla w_h]\!])\bar{\psi}_{e}\,\mathrm{d}s=\int_{{e}} ([\![ A\nabla (w_h-w)]\!])\bar{\psi}_{e}\,\mathrm{d}s\\
=&\int_{\omega_e}A\nabla (w_h-w)\cdot \overline{\nabla \psi_{e}}\,\mathrm{d}x
+ \int_{\omega_e}\nabla\cdot (A\nabla (w_h-w))\bar{\psi}_{e} \,\mathrm{d}x \\
=&\int_{\omega_e}A\nabla (w_h - w) \cdot \overline{\nabla \psi_{e}}\,\mathrm{d}x
+ \int_{\omega_e}n(x)(\lambda-1)  w  \bar{\psi}_{e}\,\mathrm{d}x
+ \int_{\omega_e}\nabla\cdot (A\nabla w_h)\bar{\psi}_{e} \,\mathrm{d}x \\
\lesssim & \left(h_{e}^{-1}\|A\nabla (w_h - w)  \|_{0,\omega_{e}}+\|
n(x)(\lambda-1) w  + \nabla\cdot (A\nabla w_h)\|_{0,\omega_{e}}\right)
\|\bar{\psi}_{e}\|_{0,\omega_{e}}.
\end{aligned}
\]
Using (\ref{s3.17}) and (\ref{s3.18}), we have
\[
\begin{aligned}
&\|[\![ A\nabla w_h]\!]\|_{0,e}^2\\
\lesssim & \left(h_{e}^{-1/2}\|A\nabla (w_h - w)  \|_{0,\omega_{e}}+h_{e}^{1/2}\|n(x)(\lambda-1) w  + \nabla\cdot (A\nabla w_h)\|_{0,\omega_{e}}\right)
\|[\![ A\nabla w_h]\!]\|_{0,e}.\\
\end{aligned}
\]
Thus,
\begin{eqnarray*}
&&h_{e}^{1/2}\|[\![ A\nabla w_h]\!]\|_{0,e}
\lesssim\|A\nabla (w_h - w) \|_{0,\omega_{e}}+h_{e}\|n(x)(\lambda_{h}-1) w_{h}  + \nabla\cdot (A\nabla w_h)\|_{0,\omega_{e}}\nonumber\\
&&~~~~~~+h_{e}\|n(x)(\lambda-1) w -n(x)(\lambda_{h}-1) w_{h} \|_{0,\omega_{e}}.
\end{eqnarray*}
Combining the above estimate and (\ref{s3.20}) we obtain
\begin{eqnarray}\label{s3.22}
h_{e}^{1/2}\|[\![ A\nabla w_h]\!]\|_{0,e}
\lesssim \|A\nabla (w_h - w) \|_{0,\omega_{e}}+
h_{e}\|n(x)(\lambda-1) w -n(x)(\lambda_{h}-1) w_{h} \|_{0,\omega_{e}}.\nonumber\\
\end{eqnarray}
Similarly, we can deduce
\begin{eqnarray}\label{s3.23}
h_{e}^{1/2}\|[\![ \nabla v_h]\!]\|_{0,e}
\lesssim \|\nabla (v_h - v) \|_{0,\omega_{e}}+
h_{e}\|(\lambda-1) v -(\lambda_{h}-1)v_{h} \|_{0,\omega_{e}}.\nonumber\\
\end{eqnarray}
(iii)~ For $e\in\mathcal{E}_h^b$,
let $\psi_{e}$ with $\mathrm{supp}(\psi_{e})\subset\omega_{e}$ be an extension of $\mathfrak{b}_{e}( A\nabla w_h- \nabla v_h )|_e$ satisfying
(\ref{s3.17}) and (\ref{s3.18}).
Then
\[
\begin{aligned}
&\int_{{e}} (( A\nabla w_h-\nabla v_h)\cdot\nu )\bar{\psi}_{e}\,\mathrm{d}s=\int_{{e}} (( A\nabla( w_h-w)-\nabla (v_h-v))\cdot\nu )\bar{\psi}_{e}\,\mathrm{d}s\\
=&\int_{\omega_{e}}A\nabla (w_h-w)\cdot \overline{\nabla \psi_{e}}\,\mathrm{d}x
+ \int_{\omega_{e}}\nabla\cdot (A\nabla( w_h-w))\bar{\psi}_{e} \,\mathrm{d}x \\
&~~~+\int_{\omega_{e}}\nabla (v_h-v)\cdot \overline{\nabla \psi_{e}}\,\mathrm{d}x
+ \int_{\omega_{e}}\nabla\cdot (\nabla(v_h-v))\bar{\psi}_{e} \,\mathrm{d}x \\
=&\int_{\omega_{e}}A\nabla (w_h - w) \cdot \overline{\nabla \psi_{e}}\,\mathrm{d}x
+ \int_{\omega_{e}}n(x)(\lambda-1)  w  \bar{\psi}_{e}\,\mathrm{d}x
+ \int_{\omega_{e}}\nabla\cdot (A\nabla w_h)\bar{\psi}_{e} \,\mathrm{d}x \\
&~~~+\int_{\omega_{e}}\nabla (v_h - v) \cdot \overline{\nabla \psi_{e}}\,\mathrm{d}x
+ \int_{\omega_{e}}(\lambda-1)v  \bar{\psi}_{e}\,\mathrm{d}x
+ \int_{\omega_{e}}\nabla\cdot (\nabla v_h)\bar{\psi}_{e} \,\mathrm{d}x \\
\lesssim & \left( h_{e}^{-1}\|A\nabla (w_h - w)  \|_{0,\omega_e}+\|
n(x)(\lambda-1) w  + \nabla\cdot (A\nabla w_h)\|_{0,\omega_e}\right)
\|\bar{\psi}_{e}\|_{0,\omega_e}\\
&~~~+ \left( h_{e}^{-1}\|\nabla (v_h -v)  \|_{0,\omega_e}+\|
(\lambda-1)v + \nabla\cdot (\nabla v_h)\|_{0,\omega_e} \right)
\|\bar{\psi}_{e}\|_{0,\omega_e}.
\end{aligned}
\]
Using (\ref{s3.17}) and (\ref{s3.18}), we obtain
\[
\begin{aligned}
&~~~\|( A\nabla w_h-\nabla v_h)\cdot\nu \|_{0,e}^2\\
& \lesssim(h_{e}^{-1/2}\|A\nabla (w_h - w)  \|_{0,\omega_e}+h_{e}^{1/2}\|n(x)(\lambda-1) w  + \nabla\cdot (A\nabla w_h)\|_{0,\omega_e})\|( A\nabla w_h-\nabla v_h)\cdot\nu \|_{0,e},\\
&~~~+
(h_{e}^{-1/2}\|\nabla (v_h - v)\|_{0,\omega_e}+h_{e}^{1/2}\|(\lambda-1)v  + \nabla\cdot (\nabla v_h)\|_{0,\omega_e})\|( A\nabla w_h-\nabla v_h)\cdot\nu \|_{0,e}.\\
\end{aligned}
\]
Thus, by (\ref{s3.20}) and (\ref{s3.21}) we deduce
\begin{eqnarray}\label{s3.24}
&&h_{e}^{1/2}\|( A\nabla w_h-\nabla v_h)\cdot\nu \|_{0,e}\lesssim\|A\nabla (w_h - w) \|_{0,\omega_e}+\|\nabla (v_h - v) \|_{0,\omega_e}\nonumber\\
&&~~~+h_{e}\|n(x)(\lambda-1) w  -n(x)(\lambda_{h}-1) w_h)\|_{0,\omega_{e}}\nonumber\\
&&~~~+h_{e}\|(\lambda-1)v -(\lambda_{h}-1) v_h)\|_{0,\omega_e}).
\end{eqnarray}
The proof is completed by substituting (\ref{s3.20})-(\ref{s3.24}) into (\ref{s3.1}).
\end{proof}

\section{Numerical experiments}\label{sec:sec4}
\indent Using the a posteriori error estimators in this paper and consulting the existing standard algorithms (see, e.g., \cite{dai1}), we present the following algorithm.\\
\noindent{\bf Algorithm 1}\\
\indent Choose the parameter $\theta\in(0,1)$.\\
\noindent{\bf Step 1.} Set $l=0$ and pick any initial mesh $\mathcal{T}_{h_l}$ with the mesh size $h_{l}$.\\
\noindent{\bf Step 2.} Solve (\ref{s2.21}) on $\mathcal{T}_{h_l}$ for discrete solution $\{(\lambda_{j,h_l},\mathbf{U}_{j,h_l})\}_{k}^{k+q-1}$ with
$\|\mathbf{U}_{j,h_l}\|_{\mathbb{V}}$
$=1$.\\
\noindent{\bf Step 3.} Compute the local estimators
$ \eta_{\kappa}(\mathbf{U}_{j,h_{l}})~(j=k,\cdots,k+q-1)$.\\
\noindent{\bf Step 4.} Construct $\hat{\mathcal{T}}_{h_l}\subset \mathcal{T}_{h_l}$ by {\bf Marking strategy E}.\\
\noindent{\bf Step 5.} Refine $\mathcal{T}_{h_l}$ to get a new mesh $\mathcal{T}_{h_{l+1}}$ by procedure {\bf Refine}.\\
\noindent{\bf Step 6.} $l\Leftarrow l+1$ and goto Step 2.\\

\noindent{\bf Marking Strategy E}\\
\indent Given parameter  $\theta\in(0,1)$.\\
\noindent{\bf Step 1.} Construct a minimal subset
$\widehat{\mathcal{T}}_{h_{l}}$ of $\mathcal{T}_{h_{l}}$ by selecting some elements
in $\mathcal{T}_{h_{l}}$ such that
\begin{equation*}
 \sum\limits_{\kappa\in\widehat{\mathcal{T}}_{h_{l}}}\sum\limits_{j=k}^{k+q-1}\eta^2_{\kappa}(\mathbf{U}_{j,h_{l}})
\geq
\theta\sum\limits_{\kappa\in\mathcal{T}_{h_l}}\sum\limits_{j=k}^{k+q-1}\eta^2_{\kappa}(\mathbf{U}_{j,h_{l}}).
\end{equation*}
\noindent{\bf Step 2.} Mark all the elements in
$\widehat{\mathcal{T}}_{h_{l}}$.\\
\noindent The above marking strategy was introduced by D$\ddot{o}$rfler \cite{dorfler}.
Our algorithm is easily realized under the common packages of the FEM, e.g., \cite{ifem,scikit-fem}, etc.\\
\indent Next we will provide some numerical examples to verify the theoretical convergence rates of our adaptive algorithm.\\
\indent Our program is partly completed under the Python package of scikit-fem \cite{scikit-fem} (version 5.2.0), then the discrete algebraic eigenvalue problems are solved by the command 'eigs' of MATLAB 2021b on a Lenovo xiaoxin Pro13.3 laptop with 16G memory.\\
\indent For discretizations, we use the standard Lagrange finite elements.
Let $DoF$ denote the number of degrees of freedom.
Let $S_h\subset H^1(\Omega)$ be the finite element space of degree $m$ on $\Omega$, and let $S_h^0=S_h\cap H_0^1(\Omega)$ and $S_h^B$ denote the subspace of functions in $S_h$ with vanishing $DoF$ on $\partial\Omega$ and the subspace of functions in $S_h$ with vanishing $DoF$ in $\Omega$, respectively.\\
\indent Let $N_h = \dim{(S_h)}, N_h^0=\dim{(S_h^0)}$ and $N_h^B=\dim{(S_h^B)}$. Let $\{\xi_i\}_{i=1}^{N_h}$ be a basis of $S_h$.
We set $\xi_i^0=\xi_i~(i=1,\cdots,N_h^0)$ and let $\{\xi_i^0\}_{i=1}^{N_{h}^0}$ be a basis of $S^{0}_h$,
and set $\xi_j^B=\xi_i~(i=N_h^0+1,\cdots,N_h,j=1,\cdots,N_h^B)$ and let $\{\xi_i^B\}_{i=1}^{N_h^B}$ be a basis of $S^{B}_h$, then for any $(w_h,v_h)\in\mathbb{V}_h$,
$$w_h = \sum_{i=1}^{N_{h}^0}w_i\xi_i + \sum_{i=1}^{N_{h}^B}w_{N_h^0+i}\xi_i,\quad
v_h = \sum_{i=1}^{N_h^0}v_i\xi_i  + \sum_{i=1}^{N_{h}^B}w_{N_h^0+i}\xi_i.$$
Denote $\boldsymbol{w}^0=(w_1,\cdots,w_{N_{h}^0})$,
$\boldsymbol{w}^B=(w_1,\cdots,w_{N_{h}^B})$
and
$\boldsymbol{v}^0=(v_1,\cdots,v_{N_{h}^0})$.\\
We specify the following matrices in the discrete case.
\begin{center} \footnotesize
\begin{tabular}{lllll}\hline
Matrix&Dimension&Definition\\\hline
$S_A$&$N_h\times N_h$&$s_{li}=\int_{\Omega} (A\nabla\xi_{i})\cdot \nabla\xi_{l}\,\mathrm{d}x$\\
$S$&$N_{h}\times N_{h}$&$s_{li}=\int_{\Omega}\nabla\xi_{i}\cdot \nabla\xi_{l}\,\mathrm{d}x$\\
$M_n$&$N_{h}\times N_{h}$&$m_{li}=\int_{\Omega} n(x)\xi_i\xi_l\,\mathrm{d}x$\\
$M$&$N_{h}\times N_{h}$&$m_{li}=\int_{\Omega}  \xi_i\xi_l\,\mathrm{d}x$\\
\hline
\end{tabular}
\end{center}
\bigskip
Then the discrete variational form (\ref{s2.21})
can be written as a generalized matrix eigenvalue problem:
\begin{equation}\label{s4.1}
(\mathcal{S}+\mathcal{M})\boldsymbol{x} = \lambda_h\mathcal{M}\boldsymbol{x}
\end{equation}
where $\boldsymbol{x}=(\boldsymbol{w}^0,\boldsymbol{v}^0,\boldsymbol{w}^B)^T$ and the matrices $\mathcal{S}$ and $\mathcal{M}$ are given by
\begin{eqnarray*}
\mathcal{S}=
\left(
\begin{array}{ccc}
{S}_A^{N_h^0\times N_h^0}&\quad\boldsymbol{0}\quad&{S}_A^{N_h^0\times N_h^B}\\
\boldsymbol{0}&\quad-{S}^{N_h^0\times N_h^0}\quad&-{S}^{N_h^0\times N_h^B}\\
{S}_A^{N_h^B\times N_h^0}&\quad-{S}^{N_h^B\times N_h^0}\quad&{S}_A^{N_h^B\times N_h^B}-{S}^{N_h^B\times N_h^B}\\
\end{array}
\right)
\end{eqnarray*}
and
\begin{eqnarray*}
\mathcal{M}=
\left(
\begin{array}{ccc}
{M}_n^{N_h^0\times N_h^0}&\quad\boldsymbol{0}\quad&{M}_n^{N_h^0\times N_h^B}\\
\boldsymbol{0}&\quad-{M}^{N_h^0\times N_h^0}\quad&-{M}^{N_h^0\times N_h^B}\\
{M}_n^{N_h^B\times N_h^0}&\quad-{M}^{N_h^B\times N_h^0}\quad&{M}_n^{N_h^B\times N_h^B}-{M}^{N_h^B\times N_h^B}\\
\end{array}
\right).
\end{eqnarray*}
\indent In our computation, the test domains are set to be the unit square $\Omega = (0,1)^2$
 and the L-shaped domain $\Omega = (-1,1)^2\backslash [0,1)\times (-1,0]$
for the two-dimensional cases and the Fichera domain $(-1,1)^3\backslash [0,1)^3$ for the three-dimensional case.
The coefficient matrix $A(x)$ and the index of refraction $n(x)$ are chosen as follows:
$$\mathrm{Case~1:~}A(x) = \frac{1}{8}\left(
  \begin{array}{cc}
    4+x_1^2 & x_1x_2\\
    x_1x_2 & 4+x_2^2\\
  \end{array}
\right),\quad n(x)=\frac{1}{4}+\frac{1}{8}(x_1+x_2);$$
$$\mathrm{Case~2:~}A(x) = \left(
  \begin{array}{cc}
    2+x_1^2 & x_1x_2\\
    x_1x_2 & 2+x_2^2\\
  \end{array}
\right),\quad n(x)=2+|x_1+x_2|;$$
$$\mathrm{Case~3:~}A(x) = \left(
  \begin{array}{ccc}
    2& 0&0\\
    0& 3&0\\
    0& 0&4\\
  \end{array}
\right),\quad n(x)=3.$$

\indent We use the sparse solver $eigs$ to solve the generalized matrix
eigenvalue problem (\ref{s4.1}) for eigenvalues. We denote $\mathrm{k}_{j}=\sqrt{\lambda_{j}-1}$, $\lambda_{j,h_{l}}$ the $j$th eigenvalue derived from the $l$th iteration using Algorithm 1 and  $\mathrm{k}_{j,h_{l}}=\sqrt{\lambda_{j,h_{l}}-1}$,
and denote $DoF_{j,l}$ the $DoF$ for the $j$th eigenvalue after $l$ iterations in our tables and figures. For comparison, we also denote $\lambda_j^{h}$ the $j$th eigenvalue computed on the uniform mesh and $\mathrm{k}_j^h=\sqrt{\lambda_j^h-1}$.
We take the marking parameter $\theta=0.5$ for the two-dimensional cases and  $\theta=0.25$ for the three-dimensional case, respectively.\\
\indent For the two-dimensional computation cases,
we use Algorithm 1 using the $P_{4}$ element to compute the problem on triangle meshes, and the numerical results are shown in Tables 1-2.
Comparing the results in Tables 1-2 with those in Tables 4-5 of \cite{xie2017},
we can see that with our adaptive method by using high order elements the same accurate approximations are obtained by fewer $DoF$.\\
\indent Since the exact eigenvalues of the problem on all test domains are unknown,
in order to investigate the convergence behavior, on the square we use the $P_4$ element to compute both real and complex eigenvalues and get $\mathrm{k}_1\approx2.6785666416746796$ with $DoF=23346$ after $21$ adaptive iterations
and $\mathrm{k}_5\approx 5.8251046826673-0.8502179043080 \mathrm{i}$ with $DoF=31666$ after $19$ adaptive iterations as the reference values for Case 1.
On the L-shaped domain, we take $\mathrm{k}_1\approx0.8739706737685$ with $DoF=44978$ after $31$ adaptive iterations and
$\mathrm{k}_6\approx3.0448394512079-0.0824124793275\mathrm{i}$ with $DoF=78818$ after $24$ adaptive iterations as the reference values for Case 2.\\
\indent Starting with the initial mesh given in Fig. 1, some adaptive refined meshes are shown in Figs. 2-5 and the curves of the absolute error of numerical eigenvalues are depicted in Figs. 6-7 using Lagrange elements of degree $m~(m=1,2,3,4)$ on triangle meshes. From Figs. 2-5 we can see that the singularities or less regularities of the eigenfunctions are mainly around the corners.
From Figs. 6-7 it can be seen that the error curves of the real eigenvalues $\mathrm{k}_{1,h_l}$ are basically parallel to a line with slope $-m$ using Lagrange elements of degree $m~(m=1,2,3,4)$, which indicate the adaptive algorithm can reach the optimal convergence order $O({DoF}^{-2m/d})$.
We also observe from Figs. 6-7 that the accuracy of the numerical eigenvalues on adaptive meshes is better than that on uniform and quasi-uniform meshes.\\
\indent For the three-dimensional case, we use Algorithm 1 with the $P_{2}$ element to compute the first real eigenvalue of the problem on tetrahedral meshes, and the numerical results are listed in Table 3.
To depict the error curves, we use the $P_{2}$ element with $DoF=187504$ to get $\mathrm{k}_1\approx1.06768887$  as a reference value for Case 3 after $26$ adaptive iterations. 
It is easy to observe from Fig. 9  that the approximations of eigenvalue reach the optimal convergence order.

\begin{figure*}
\begin{minipage}[t]{0.50\textwidth}
\centering
\includegraphics[width=2.3in]{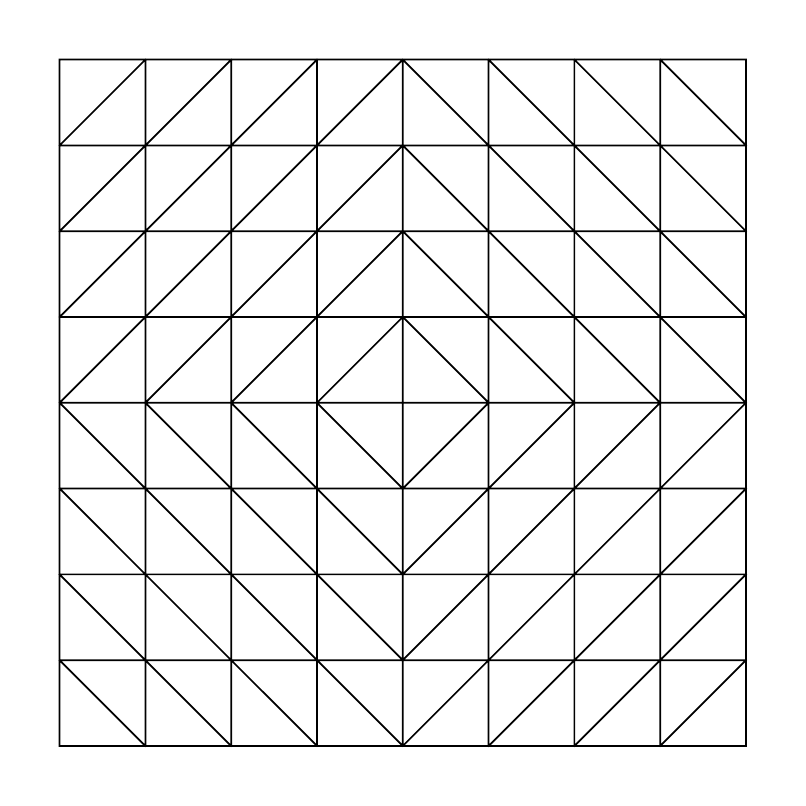}
\label{fig:side:a}
\end{minipage}%
\begin{minipage}[t]{0.50\textwidth}
\centering
\includegraphics[width=2.3in]{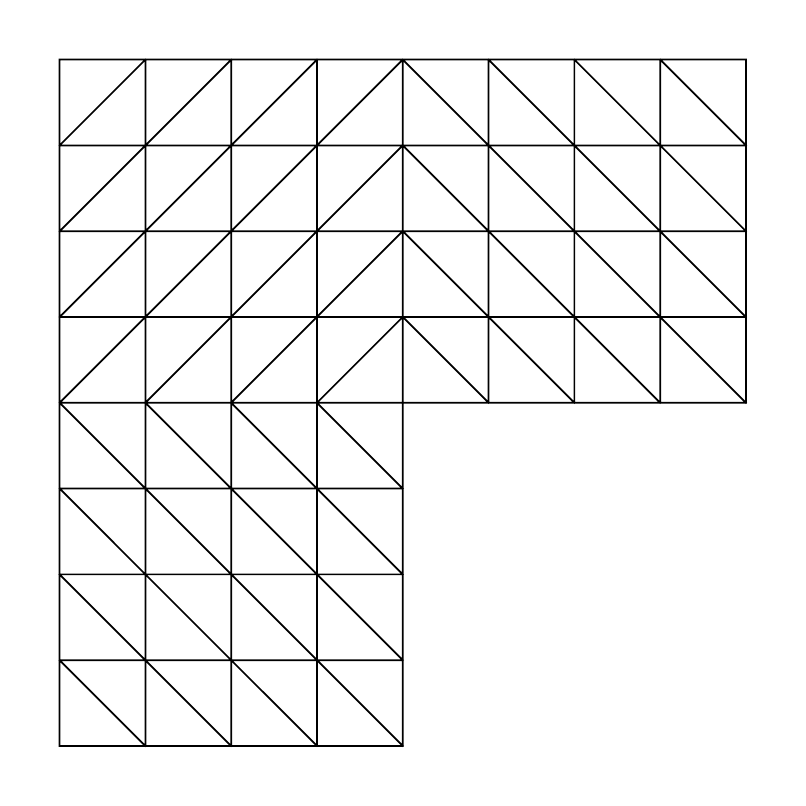}
\label{fig:side:b}
\end{minipage}
\caption{Initial mesh for adaptive method. The unit square (left) with $h=\sqrt{2}/8$ and the L-shaped domain(right) with $h=\sqrt{2}/4$.}
\end{figure*}

\begin{figure*}
\begin{minipage}[t]{0.50\textwidth}
\centering
\includegraphics[width=2.3in]{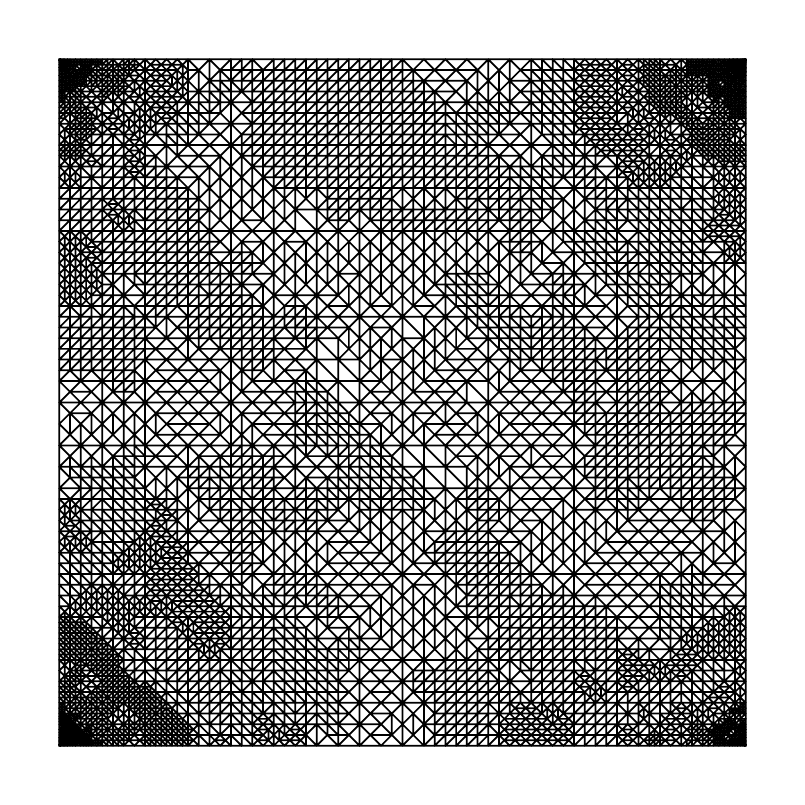}
\label{fig:side:a}
\end{minipage}%
\begin{minipage}[t]{0.50\textwidth}
\centering
\includegraphics[width=2.3in]{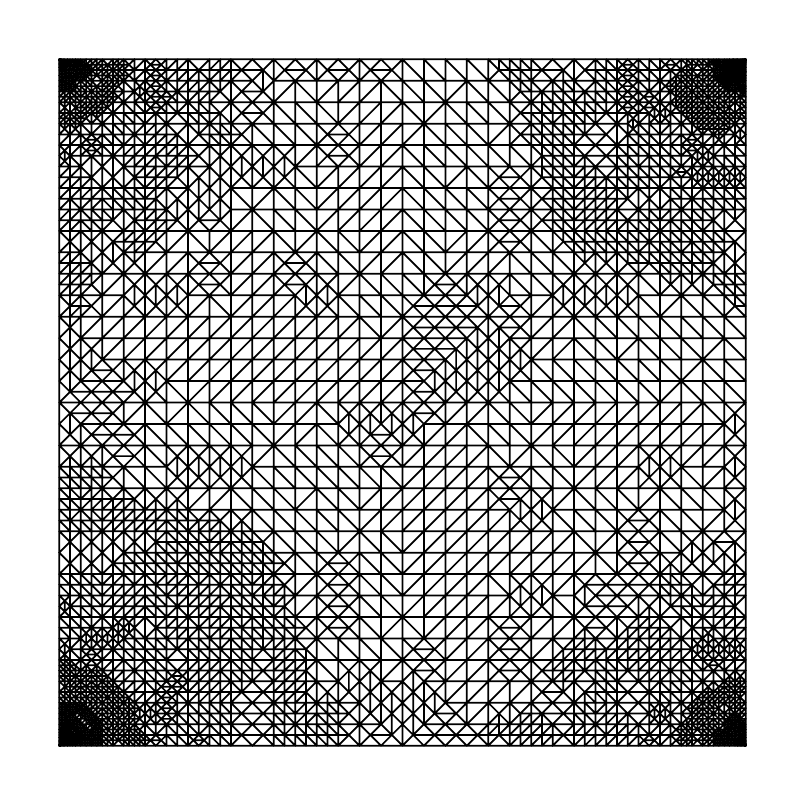}
\label{fig:side:b}
\end{minipage}
\caption{Adaptive meshes for the first eigenvalue on the unit square for Case 1 obtained by the $P_3$ element (left) and the $P_4$ element (right), respectively. }
\end{figure*}

\begin{figure*}
\begin{minipage}[t]{0.50\textwidth}
\centering
\includegraphics[width=2.3in]{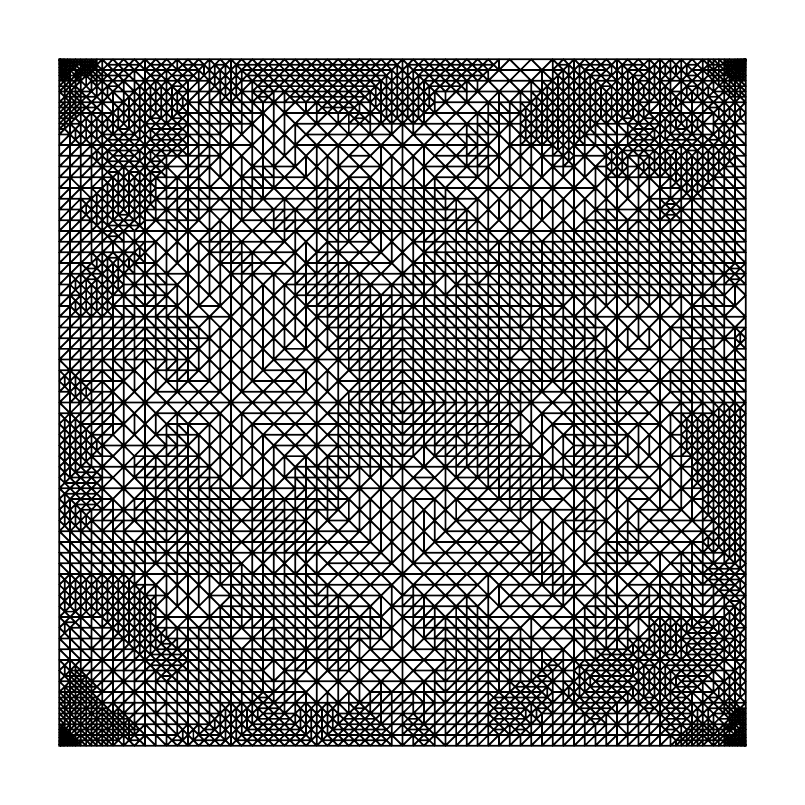}
\label{fig:side:a}
\end{minipage}%
\begin{minipage}[t]{0.50\textwidth}
\centering
\includegraphics[width=2.3in]{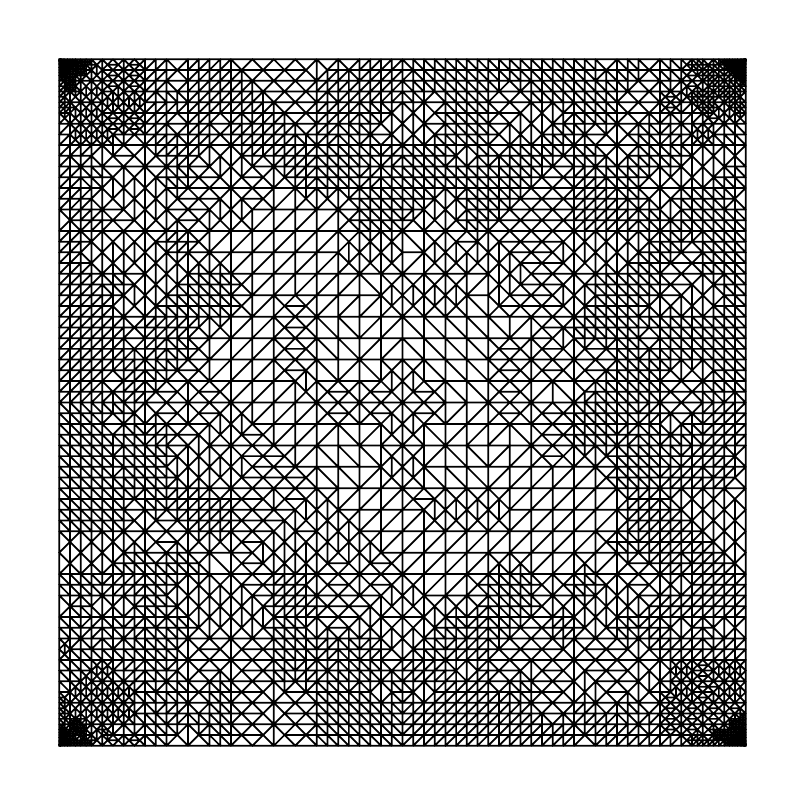}
\label{fig:side:b}
\end{minipage}
\caption{Adaptive meshes for the fifth eigenvalue on the unit square  for Case 1 obtained by the $P_3$ element (left) and the $P_4$ element (right), respectively. }
\end{figure*}

\begin{figure*}
\centering
\includegraphics[width=2.3in]{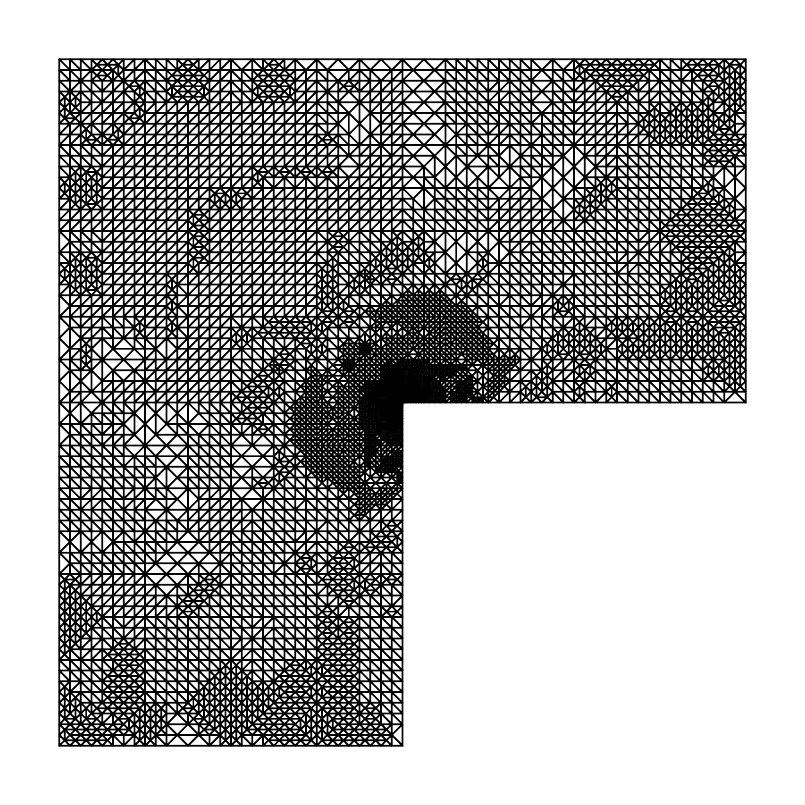}
\includegraphics[width=2.3in]{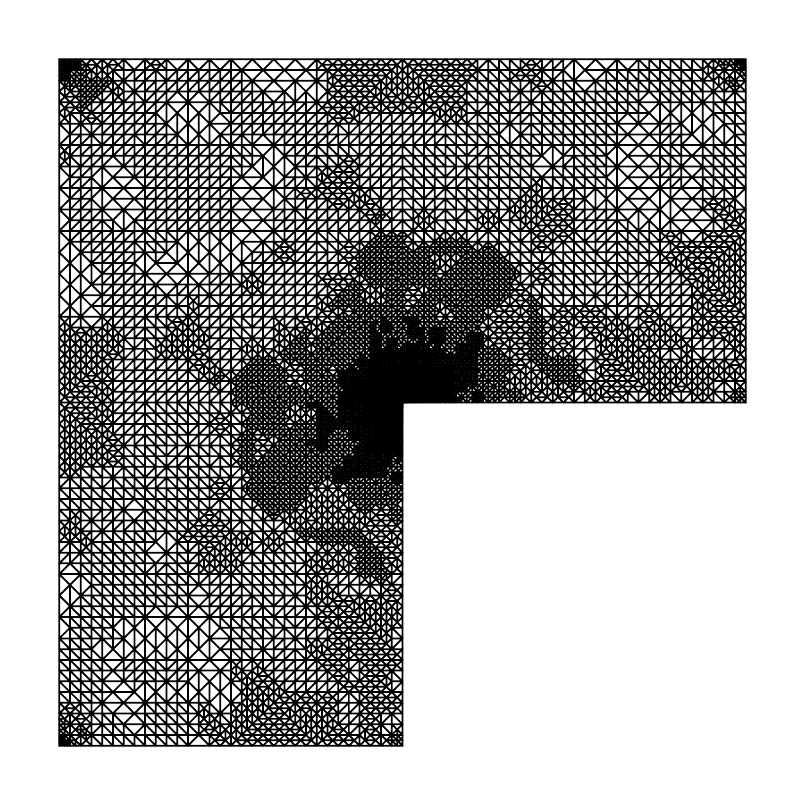}
\includegraphics[width=2.3in]{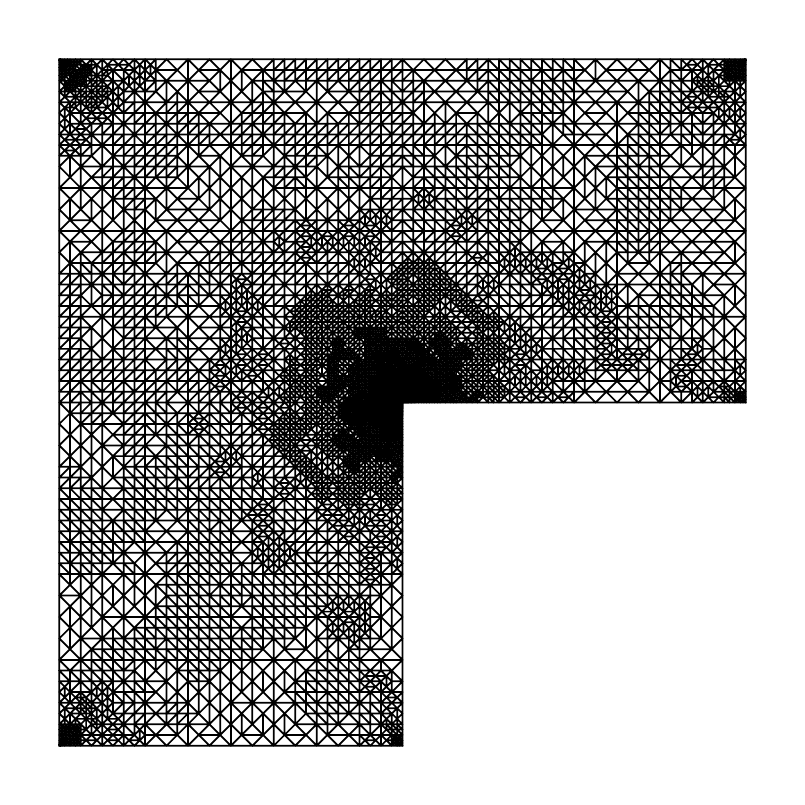}
\includegraphics[width=2.3in]{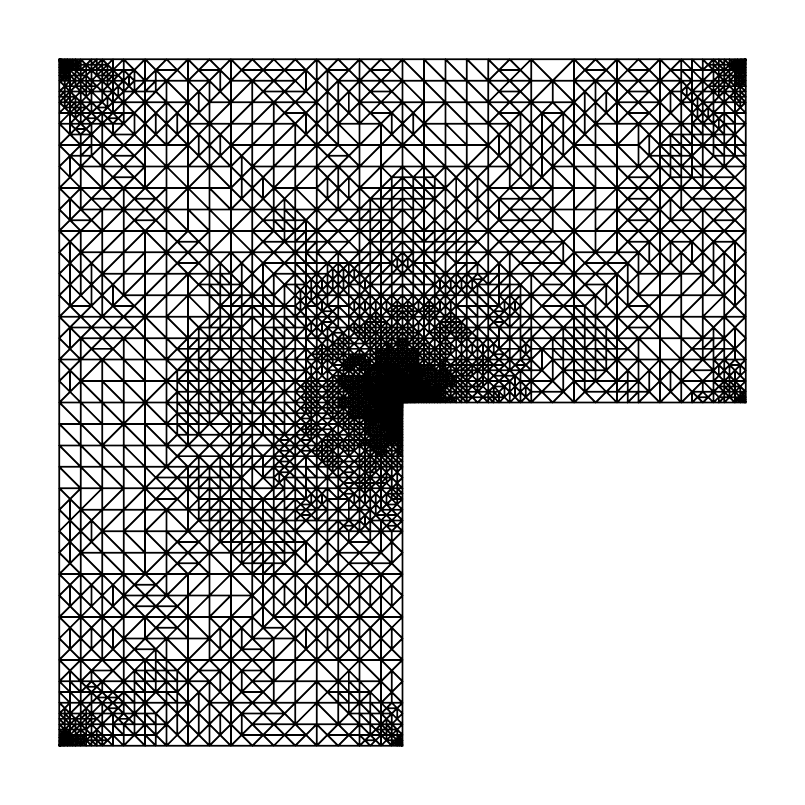}
\caption{Adaptive meshes for the first eigenvalue on the L-shaped domain for Case 2 obtained by the $P_1$ element (top left), the $P_2$ element (top right), the $P_3$ element (bottom left) and the $P_4$ element (bottom right) respectively.}
\end{figure*}

\begin{figure*}
\centering
\includegraphics[width=2.3in]{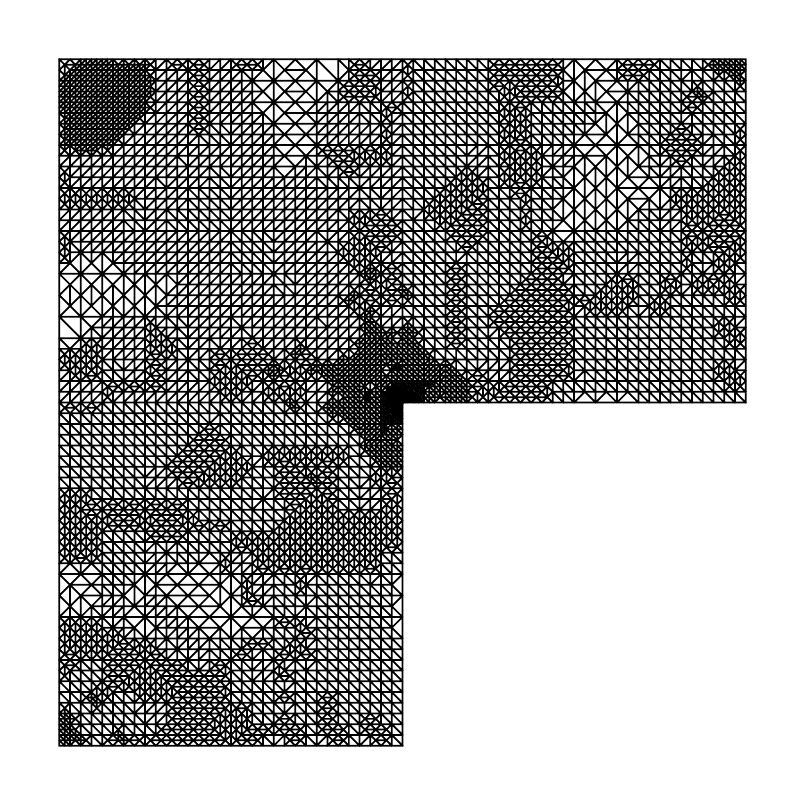}
\includegraphics[width=2.3in]{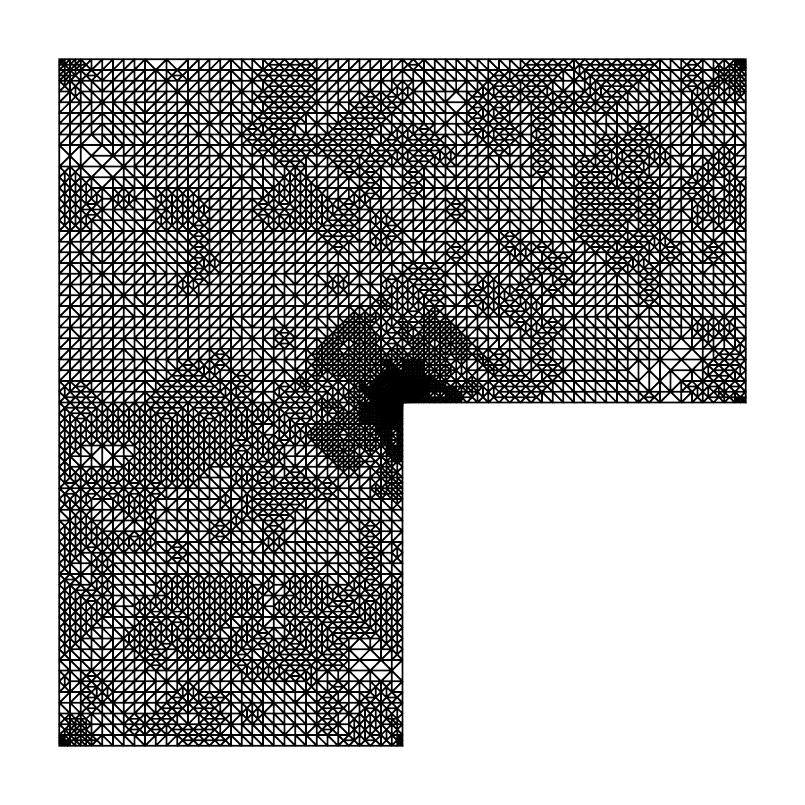}
\includegraphics[width=2.3in]{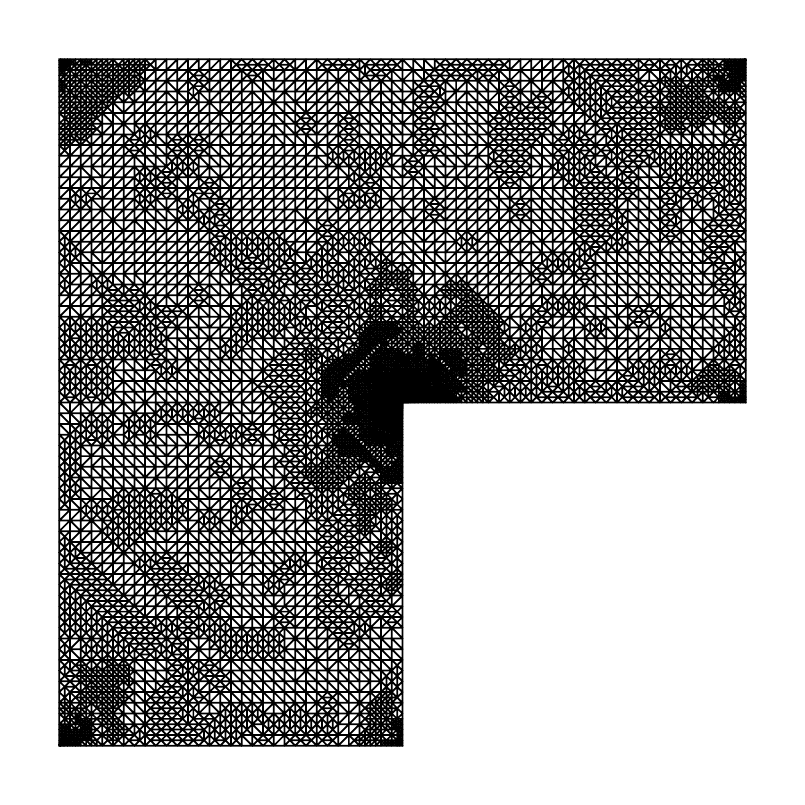}
\includegraphics[width=2.3in]{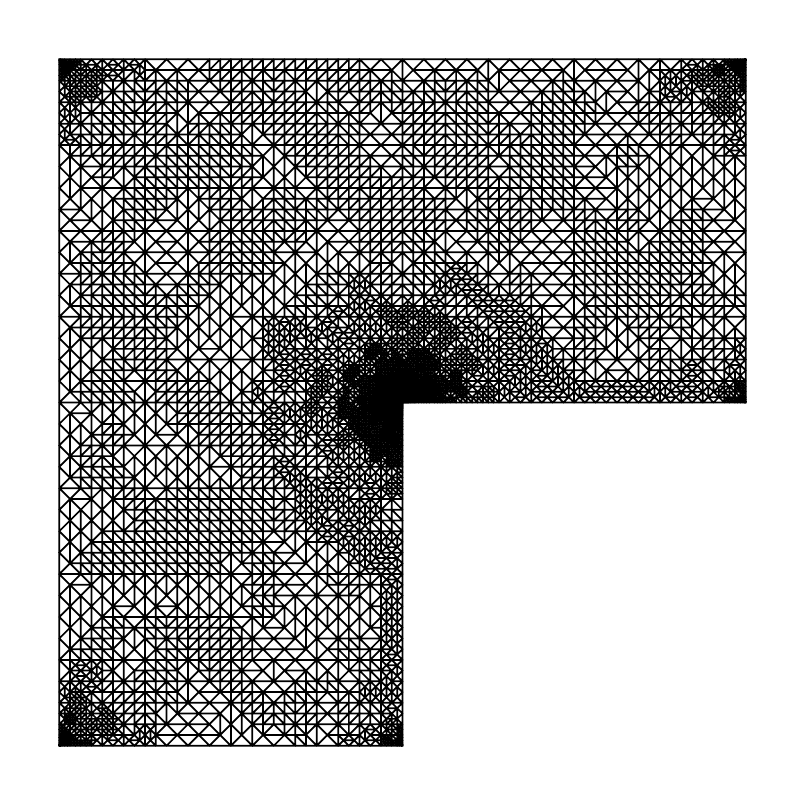}
\caption{Adaptive meshes for the sixth eigenvalue on the L-shaped domain for Case 2 obtained  by the $P_1$ element (top left), the $P_2$ element (top right), the $P_3$ element (bottom left) and the $P_4$ element (bottom right) respectively. }
\end{figure*}

\begin{figure*}
\begin{minipage}[t]{0.50\textwidth}
\centering
\includegraphics[width=2.65in]{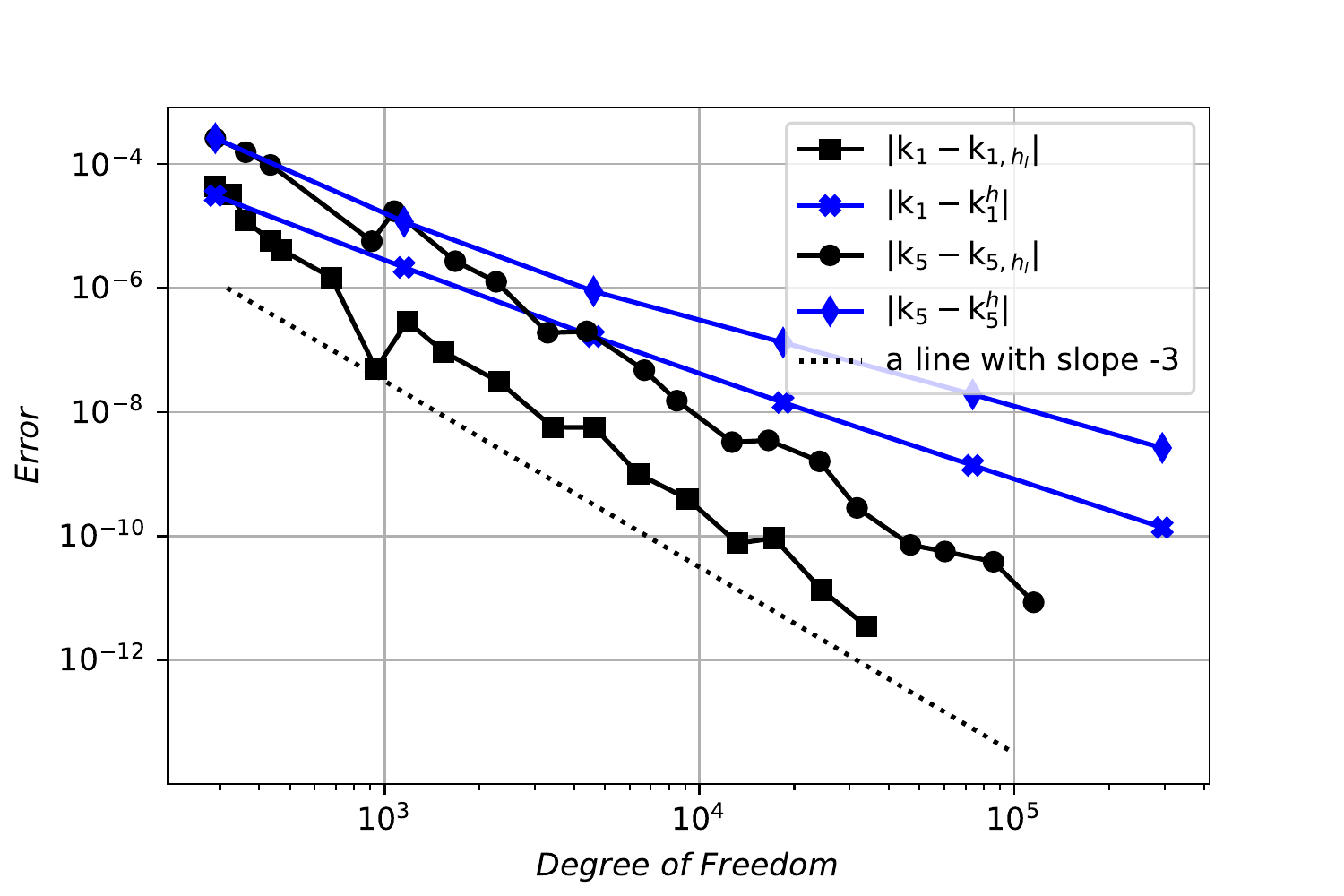}
\label{fig:side:a}
\end{minipage}%
\begin{minipage}[t]{0.50\textwidth}
\centering
\includegraphics[width=2.65in]{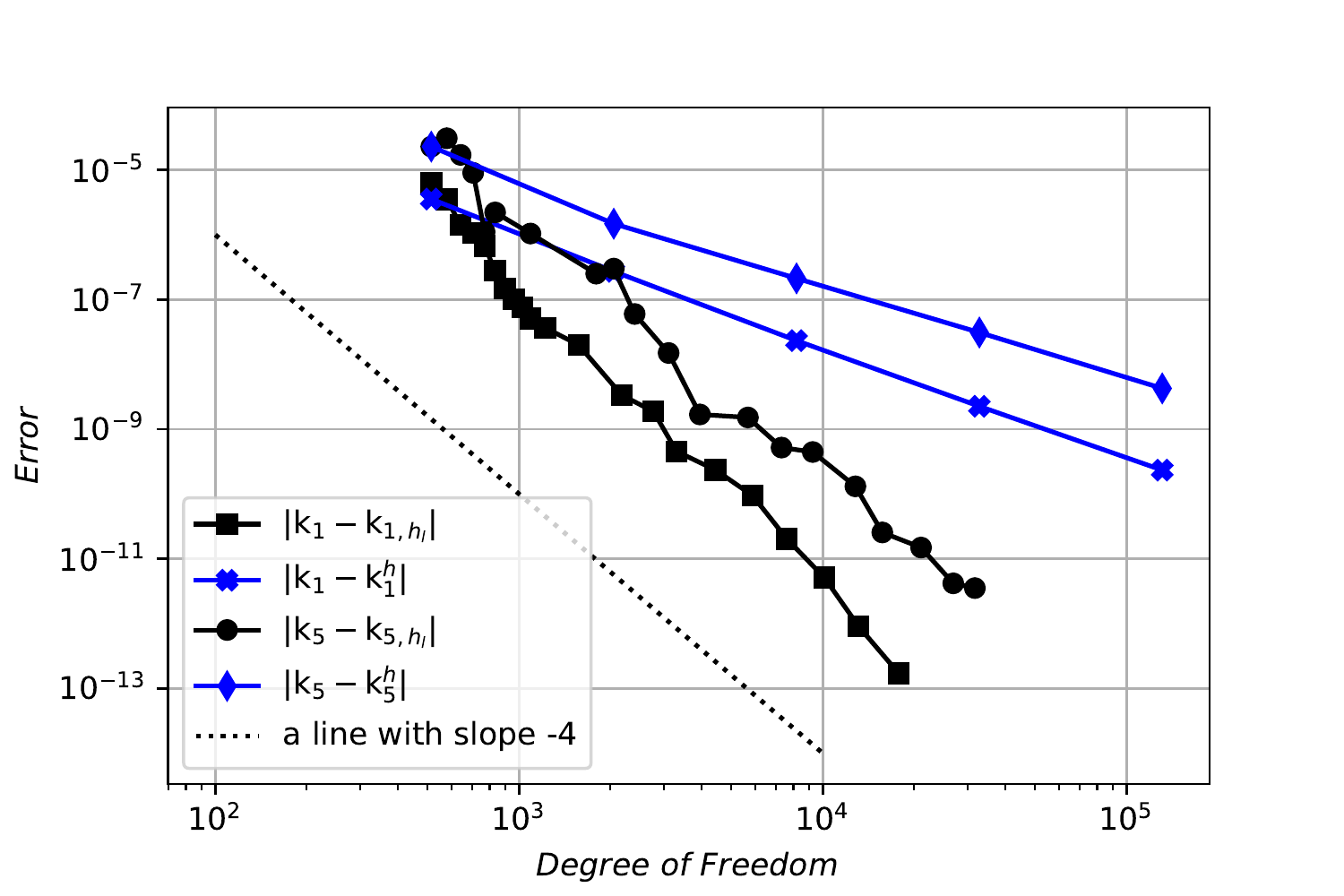}
\label{fig:side:b}
\end{minipage}
\caption{Error curves for the first and fifth eigenvalues on the unit square for Case 1 obtained by the $P_3$ element (left) and the $P_4$ element (right), respectively.}
\end{figure*}

\begin{figure*}
\begin{minipage}[t]{0.50\textwidth}
\centering
\includegraphics[width=2.65in]{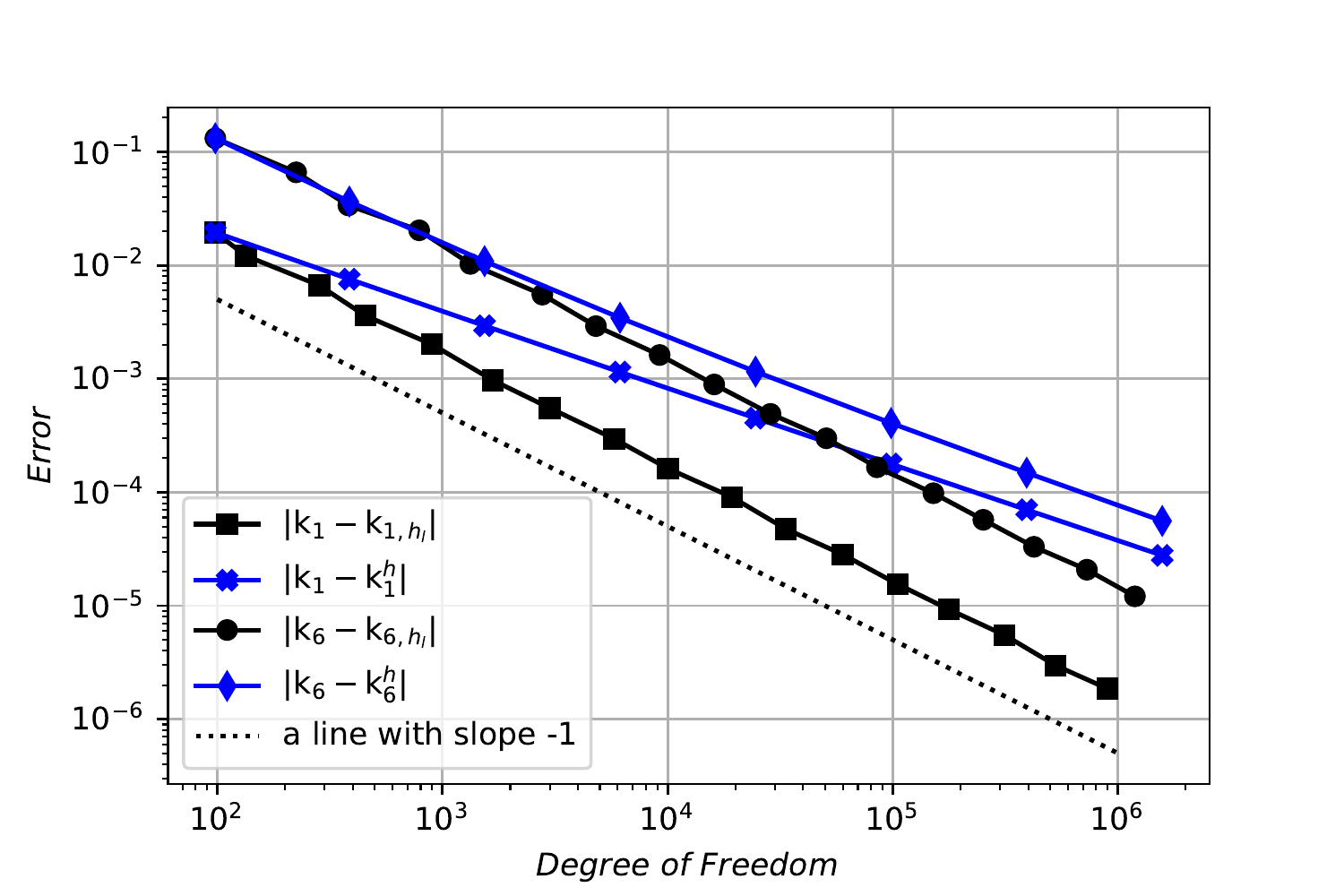}
\label{fig:side:a}
\end{minipage}%
\begin{minipage}[t]{0.50\textwidth}
\centering
\includegraphics[width=2.65in]{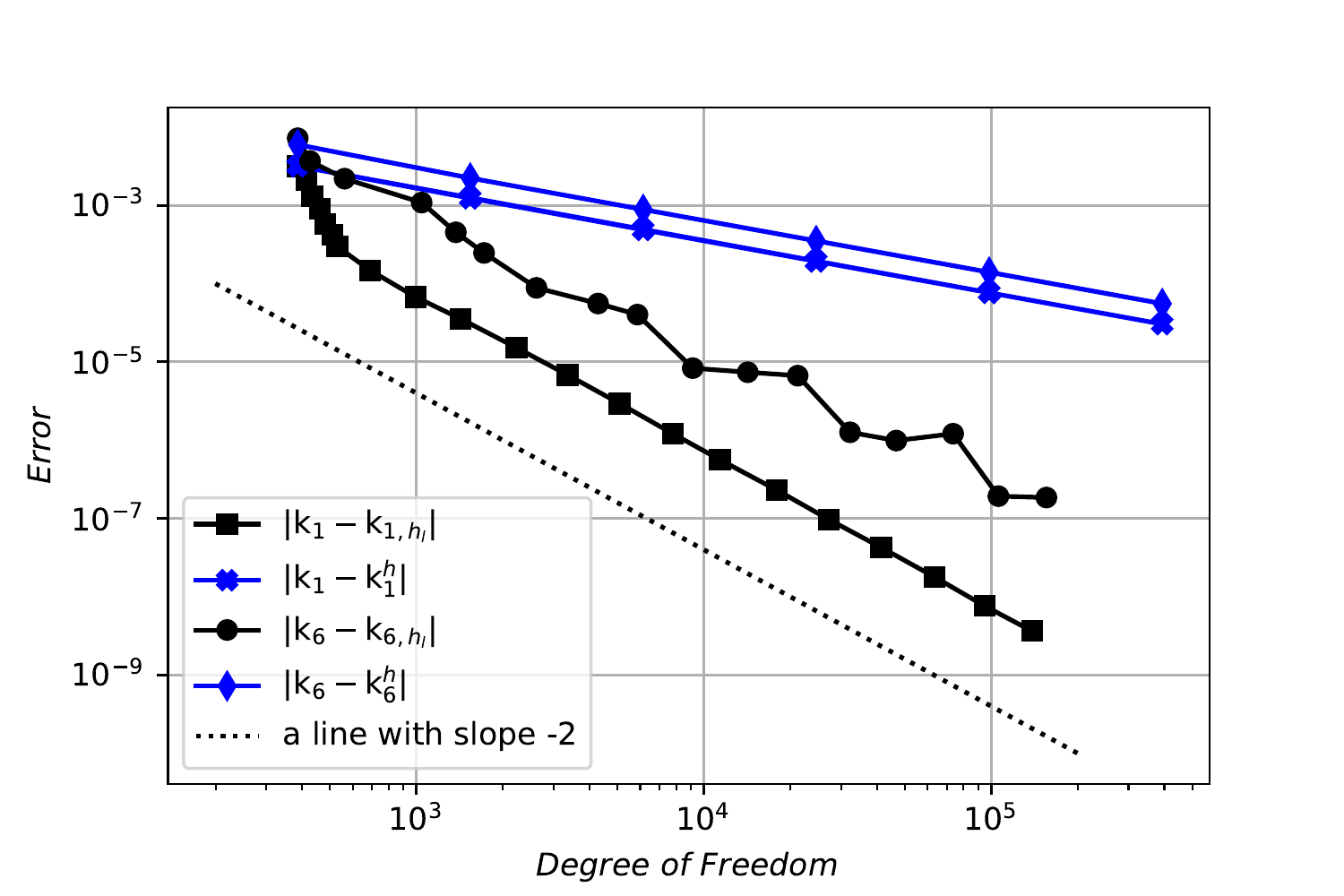}
\label{fig:side:b}
\end{minipage}
\begin{minipage}[t]{0.50\textwidth}
\centering
\includegraphics[width=2.65in]{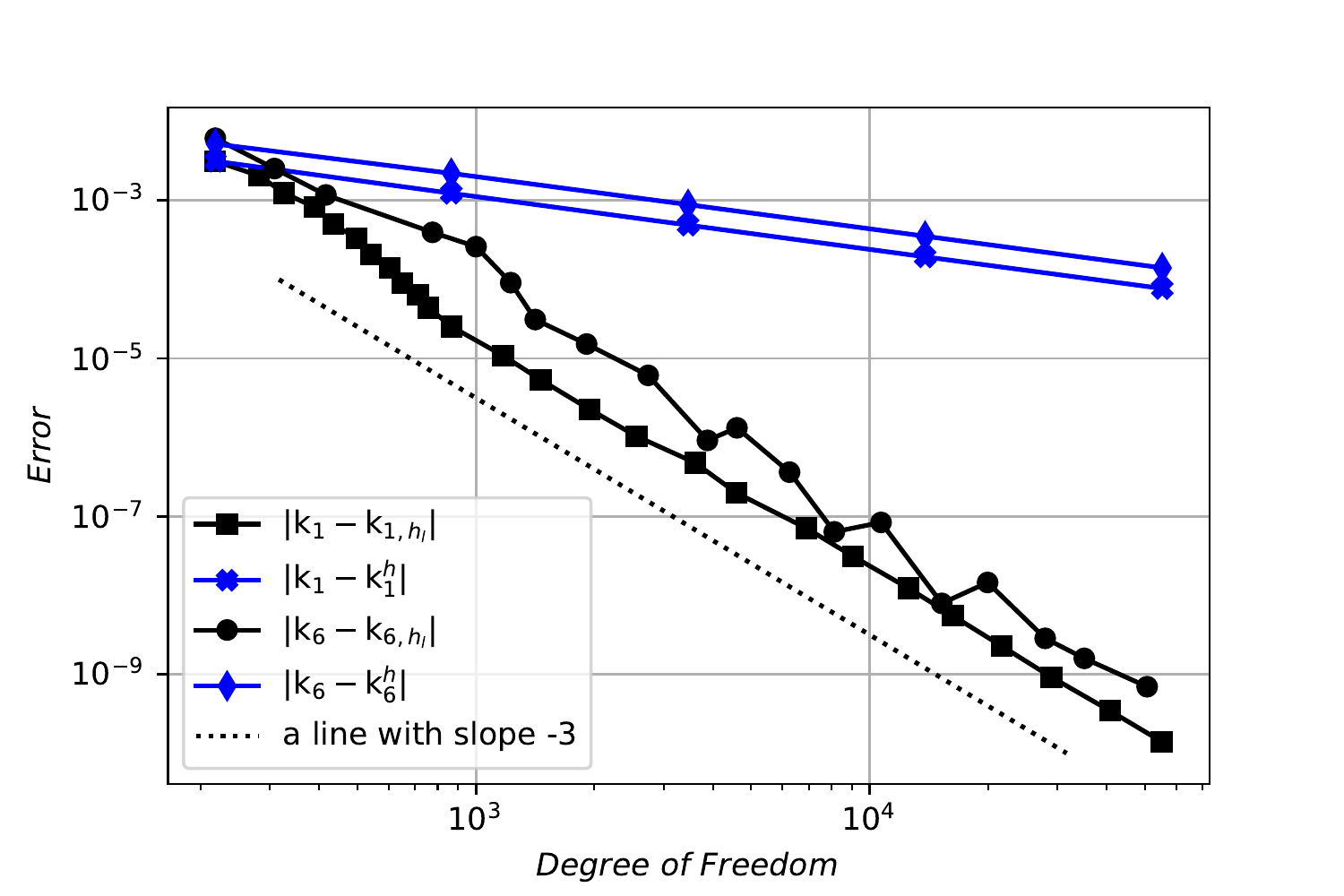}
\label{fig:side:a}
\end{minipage}%
\begin{minipage}[t]{0.50\textwidth}
\centering
\includegraphics[width=2.65in]{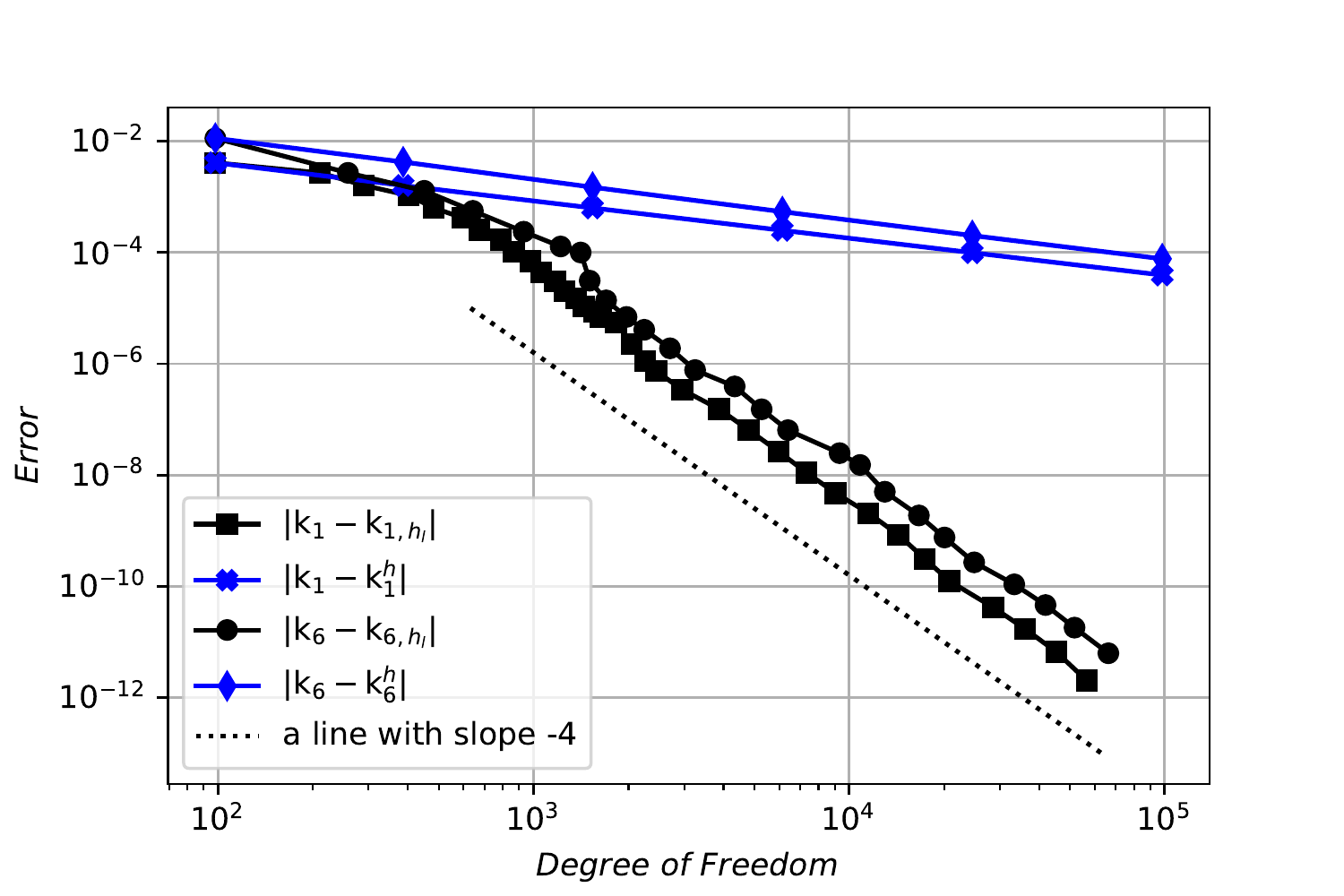}
\label{fig:side:b}
\end{minipage}
\caption{Error curves for the first and sixth eigenvalues on the L-shaped domain for Case 2 obtained by the $P_1$ element (top left), the $P_2$ element (top right), the $P_3$ element (bottom left) and the $P_4$ element (bottom right) respectively.}
\end{figure*}

\begin{table}[htb]
\caption{\label{table.label} The approximate eigenvalues on the unit square for Case 1 obtained by the $P_4$ element with the initial uniform mesh $(h_0=\sqrt{2}/4)$.}
\centering
\bigskip
\begin{small}
\begin{tabular}{cccccccc}
\hline
$j$ & $l$ & ${DoF}_{j,l}$ & $\mathrm{k}_{j,h_l}$&$j$ & $l$ & ${DoF}_{j,l}$ & $\mathrm{k}_{j,h_l}$                                              \\ \hline
1   &0&514&2.678560407762189    &5,6   &0&514    &5.8250906361254$\pm$0.8501997542054$\mathrm{i}$ \\
1   &5&834&2.6785663637974375   &5,6   &5&834    &5.8251047037874$\pm$0.8502156894591$\mathrm{i}$ \\
1   &10&1218&2.6785666048173606 &5,6   &10&3106  &5.8251046903523$\pm$0.8502179171864$\mathrm{i}$ \\
1   &15&4434&2.678566641438936  &5,6   &13&7314  &5.8251046825744$\pm$0.8502179048201$\mathrm{i}$ \\
1   &16&5874&2.6785666415793807&5,6   &14&9266  &5.8251046828306$\pm$0.8502179047199$\mathrm{i}$ \\
1   &17&7618&2.678566641653595&5,6   &15&12802 &5.8251046826660$\pm$0.8502179044386$\mathrm{i}$ \\
1   &18&10114&2.678566641668692&5,6   &16&15698 &5.8251046826741$\pm$0.8502179042839$\mathrm{i}$ \\
1   &19&13106&2.6785666416729192&5,6   &17&21058 &5.8251046826744$\pm$0.8502179042968$\mathrm{i}$ \\
1   &20&17778&2.6785666416740033&5,6   &18&26882 &5.8251046826680$\pm$0.8502179043068$\mathrm{i}$\\
1   &21&23346&2.6785666416746796&5,6   &19&31666 &5.8251046826673$\pm$0.8502179043080$\mathrm{i}$ \\ \hline
\end{tabular}
\end{small}
\end{table}

\begin{table}[htb]
\caption{\label{table.label} The approximate eigenvalues on the L-shaped domain for Case 2 obtained by the $P_4$ element with the initial uniform mesh $(h_0=\sqrt{2}/2)$.}
\centering
\bigskip
\begin{small}
\begin{tabular}{cccccccc}
\hline
$j$ & $l$ & ${DoF}_{j,l}$ & $\mathrm{k}_{j,h_l}$ &$j$ & $l$ & ${DoF}_{j,l}$ & $\mathrm{k}_{j,h_l}$           \\ \hline
1   & 0&386   &0.8755661700754&6,7 & 0   &386    &3.0490526389229$\pm$0.0822453289680 $\mathrm{i}$ \\
1   & 5&882   &0.8741362387888&6,7 & 5&1314   &3.0467668650725$\pm$0.0815300508826 $\mathrm{i}$ \\
1   & 10&1346 &0.8739876304907&6,7 & 10&3138  &3.0448876037668$\pm$0.0824072414907$\mathrm{i}$ \\
1   & 15&1842 &0.8739736924251&6,7 &15&11074 &3.0448402133483$\pm$0.0824124057516 $\mathrm{i}$ \\
1   & 20&3586 &0.8739708340761&6,7 &19&27554&3.0448395440036$\pm$0.0824124707211$\mathrm{i}$ \\
1   & 25&10818&0.8739706756303&6,7 &20&34946&3.0448395438721$\pm$0.0824124707252$\mathrm{i}$ \\
1   & 28&22242&0.8739706738797&6,7 &21&42210&3.0448394615373$\pm$0.0824124783664$\mathrm{i}$ \\
1   & 29&29410&0.8739706738016&6,7 &22&52770&3.0448394615119$\pm$0.0824124783794$\mathrm{i}$ \\
1   & 30&35762&0.8739706737785&6,7 &23&66114&3.0448394614984$\pm$0.0824124783762$\mathrm{i}$ \\
1   & 31&44978&0.8739706737685&6,7 &24&78818&3.0448394512079$\pm$0.0824124793275$\mathrm{i}$\\ \hline
\end{tabular}
\end{small}
\end{table}


\begin{figure*}
\begin{minipage}[l]{0.50\textwidth}
\centering
\includegraphics[width=2.3in]{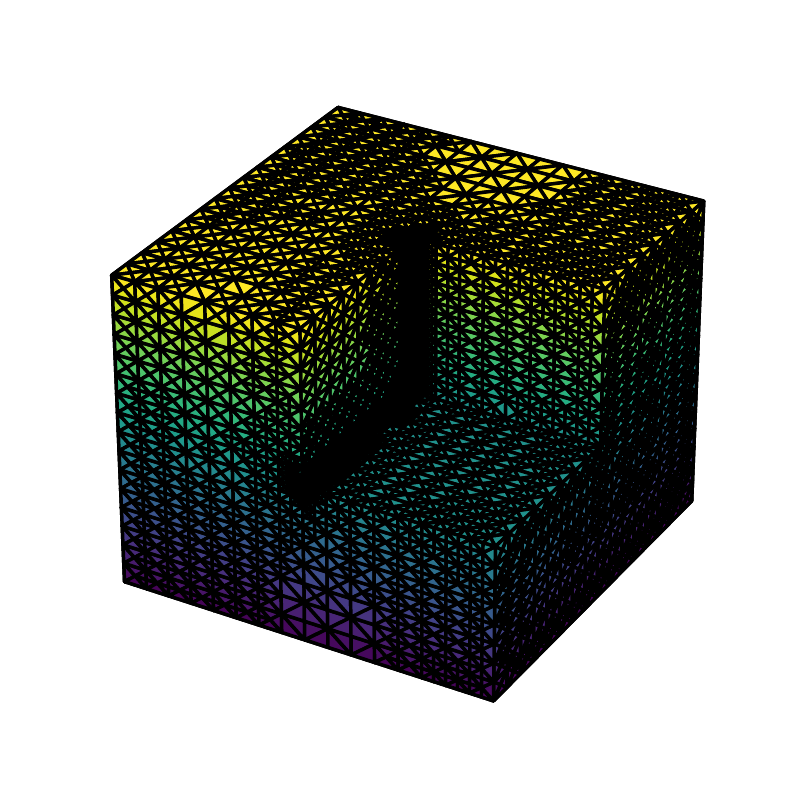}
\label{fig:side:a}
\end{minipage}%
\begin{minipage}[r]{0.50\textwidth}
\centering
\includegraphics[width=2.3in]{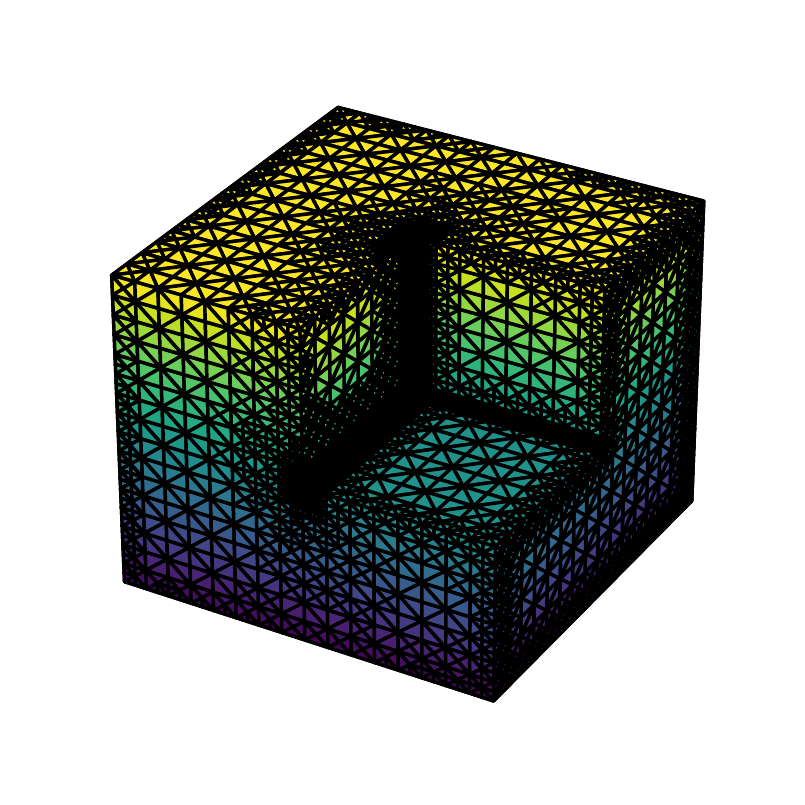}
\label{fig:side:b}
\end{minipage}
\caption{The surfaces of adaptive meshes for the first eigenvalue on the Fichera domain for Case 3 obtained  by the $P_1$ element (left) and the $P_2$ element (right), respectively}
\end{figure*}

\begin{figure*}
\begin{minipage}[l]{0.50\textwidth}
\centering
\includegraphics[width=2.3in]{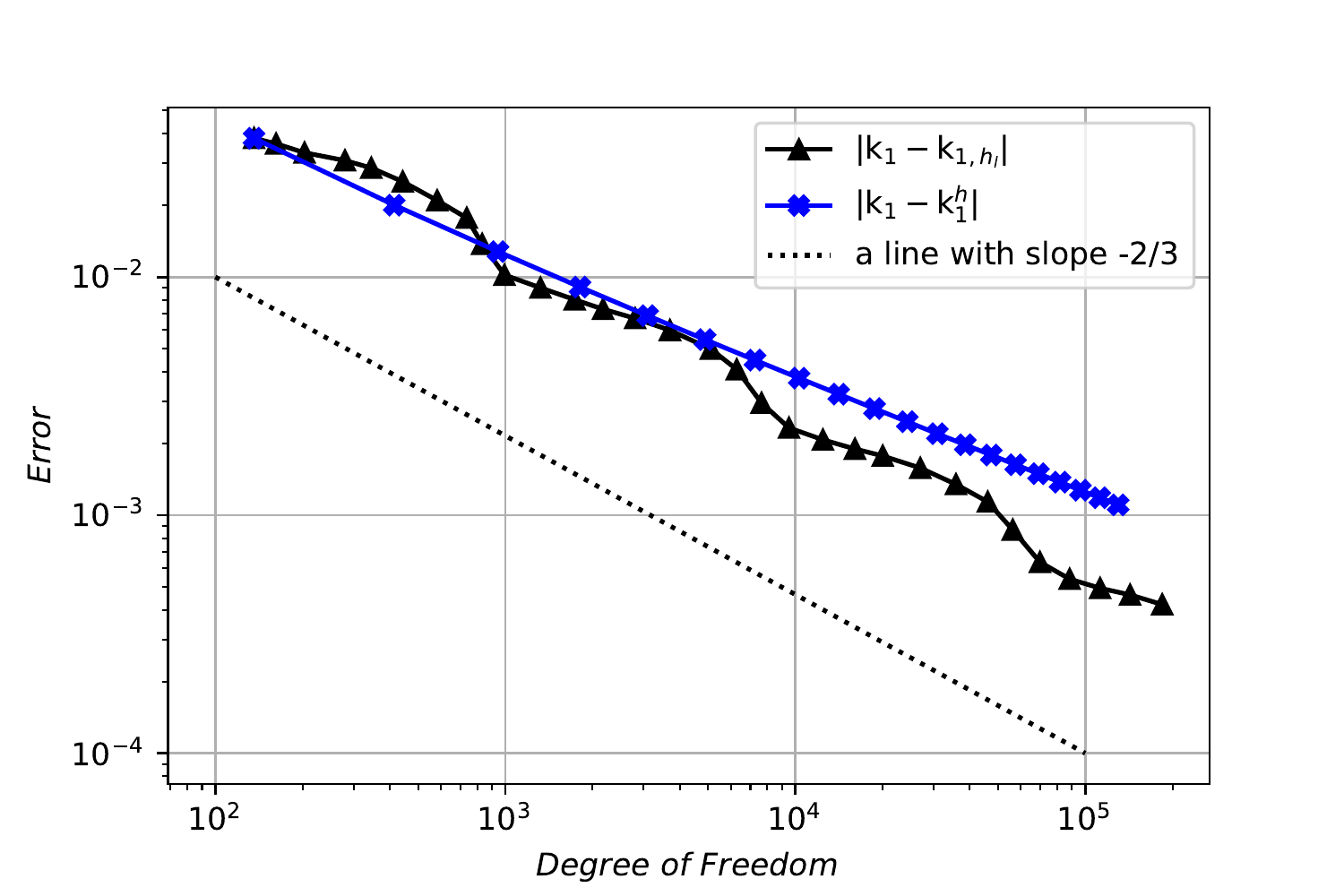}
\label{fig:side:a}
\end{minipage}%
\begin{minipage}[r]{0.50\textwidth}
\centering
\includegraphics[width=2.3in]{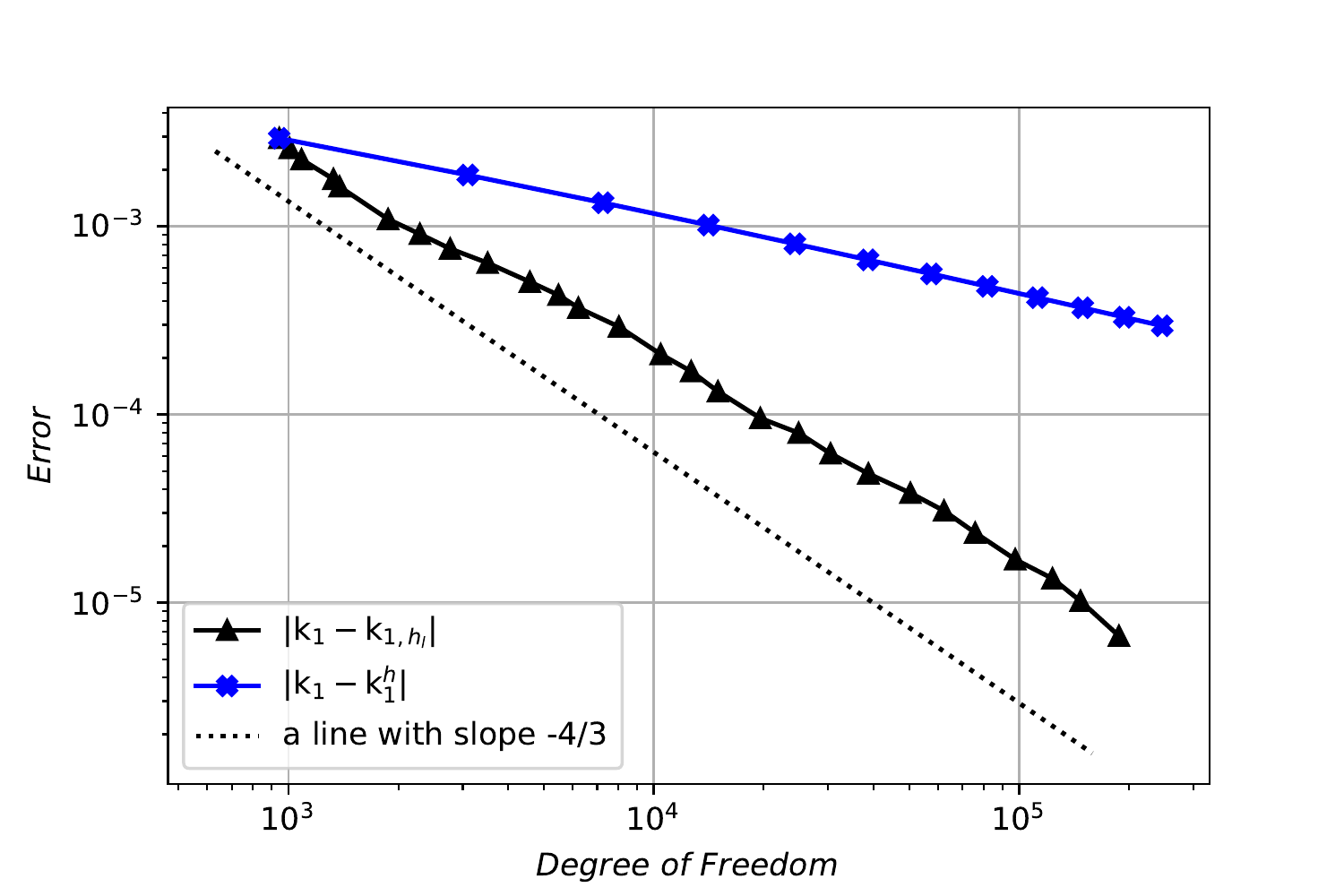}
\label{fig:side:b}
\end{minipage}
\caption{Error curves for the first eigenvalue on the Fichera domain for Case 3 obtained by the $P_1$ element (left) and the $P_2$ element (right), respectively. }
\end{figure*}

\begin{table}[htb]
\caption{\label{table.label} The first approximate eigenvalues on the Fichera domain for Case 3 obtained by the $P_2$ element with the initial uniform mesh $(h_0=\sqrt{3}/2)$.}
\centering
\bigskip
\begin{small}
\begin{tabular}{ccccccccc}
\hline
$l$ & ${DoF}_{1,l}$ & $\mathrm{k}_{1,h_l}$& $l$ & ${DoF}_{1,l}$ & $\mathrm{k}_{1,h_l}$&$l$ & ${DoF}_{1,l}$ & $\mathrm{k}_{1,h_l}$                                             \\ \hline
0 & 944   & 1.07062332 & 9  &4580  & 1.06818946& 18 &30464 & 1.06774425  \\
1 &1008   & 1.07028496 & 10 &5484  & 1.06811391& 19 &38662 & 1.06773074  \\
2 &1086  & 1.06995107  & 11 &6218  & 1.06805111& 20 &50392 & 1.06772048  \\
3 &1328  & 1.06946592  & 12 &8026  & 1.06797512& 21 &62252 & 1.06771302  \\
4 &1380  & 1.06930904  & 13 &10446 & 1.06789139& 22 &75812 & 1.06770574  \\
5 &1874  & 1.06877474  & 14 &12652 & 1.06785206& 23 &97534 & 1.06769919  \\
6 &2290  & 1.06859255  & 15 &14980 & 1.06781498& 24 &123350& 1.06769565  \\
7 &2772  & 1.06844332  & 16 &19566 & 1.06777787& 25 &147234& 1.06769242  \\
8 &3510  & 1.06832147  & 17 &24912 & 1.06776204& 26 &187504& 1.06768887  \\ \hline
\end{tabular}
\end{small}
\end{table}



\end{document}